\documentclass[11pt,a4paper,twoside]{article}
\usepackage{amssymb,amsfonts,amsmath,amsthm,euscript}
\usepackage[english]{babel}
\usepackage[latin1]{inputenc}
\usepackage{dsfont}
\usepackage{geometry}
\usepackage{graphicx}
\usepackage{multirow}
\usepackage{rotating}
\usepackage{lscape}
\usepackage[bf,footnotesize]{caption}
\usepackage{subfigure}
\usepackage{placeins}
\allowdisplaybreaks
\begin{document}

\theoremstyle{plain}
\newtheorem{thm}{Theorem}[section]
\newtheorem{cor}{Corollary}[section]
\newtheorem{prop}{Proposition}[section]
\newtheorem{lema}{Lemma}[section]
\theoremstyle{definition}
\newtheorem{rmk}{Remark}[section]
\newtheorem{df}{Definition}[section]
\renewcommand{\theequation}{\thesection.\arabic{equation}}
\renewcommand{\arraystretch}{1.35}
\newcommand{\T}{\mathcal{T}}
\renewcommand{\P}{\mathds{P}}
\newcommand{\R}{\mathds{R}}
\newcommand{\N}{\mathds{N}}
\renewcommand{\qed}{{\footnotesize$\blacksquare$}\normalsize}
\renewcommand{\phi}{\varphi}
\newcommand{\p}{\varphi}
\newcommand{\ip}{\varphi^{-1}}
\newcommand{\eps}{\varepsilon}
\newcommand{\fim}{\hfill\qed}
\newcommand{\md}{\mathrm{mod\,}}
\newcommand{\lm}{\lambda}
\renewcommand{\L}{\mathcal{L}}
\newcommand{\I}{\delta}
\renewcommand{\proof}{\noindent \textbf{\emph{Proof: }}}
\renewcommand{\cup}{\mbox{\footnotesize$\bigcup$}\,}
\renewcommand{\cap}{\mbox{\footnotesize$\bigcap$}\,}
\numberwithin{table}{section}
\numberwithin{figure}{section}
\thispagestyle{empty}
\vskip2cm
{\centering
\huge{\bf  Copulas Related to Manneville-Pomeau Processes}\vspace{.8cm}\\
\large{ {\bf S\'ilvia R.C. Lopes and Guilherme Pumi} \vspace{.1cm}\\
Mathematics Institute \\
Federal University of Rio Grande do Sul\vspace{.3cm}\\
This version: October, 10th, 2011\\
}
}
\vskip.6cm

\begin{abstract}
In this work we derive the copulas related to Manneville-Pomeau processes. We examine both bidimensional and multidimensional cases and derive some properties for the related copulas. Computational issues, approximations and random variate generation problems are also addressed and simple numerical experiments to test the approximations developed are also perform. In particular, we propose an approximation to the copula which  we show to converge uniformly to the true copula. To illustrate the usefulness of the theory, we derive a fast procedure to estimate the underlying parameter in  Manneville-Pomeau processes.
\vspace{.2cm}
\end{abstract}


\normalsize
\section{Introduction}
\setcounter{equation}{0}
The statistics of stochastic processes derived from dynamical systems has seen a grown attention in the last decade or so (see Chazottes et al. (2005) and references therein). The relationship between copulas and areas such ergodic theory and dynamical systems also have seen some development, especially in the last few years (see, for instance, Koles\'arov\'a et al. (2008)). In this work our aim is to contribute with the area by identifying and studying the copulas related to random vectors coming from the so-called Manneville-Pomeau processes, which are obtained as iterations of the Man\-ne\-vil\-le-Pomeau transformation to a specific chosen random variable (see Definitions \ref{mpmap}  and \ref{mpp}). We cover both, bidimensional and $n$-dimensional cases, which share a lot more in common than one could expect.

The copulas derived here depend on a probability measure which has no closed formula. In order to minimize this deficiency, we propose an approximation to the copula which we show to converge uniformly to the true copula. The copula also depend on several functions which have to be approximated as well, so the approximation depends on several intermediate steps. The results related to the convergence of the proposed approximation  presented here are far more general than we need and actually allows one to change these intermediate approximations and still obtain the uniform convergence result for the approximated copula. We also address problems related to random variate generation of the copula and present the results of some simple numerical experiments in order to assess the stability and precision of the intermediate approximations. The usefulness of the theory is illustrated by a simple application to the problem of estimating the underlying parameter in Manneville-Pomeau processes.

The paper is organized as follows: in the next section, we briefly review some concepts and results on Manneville-Pomeau transformations and processes and on copulas. Section 3 is devoted to determine the copulas related to any pair $(X_t,X_{t+h})$ from a Manneville-Pomeau process and to explore some consequences. In Section 4, the multidimensional extensions are shown. In Section 5 an approximation to the copulas derived in Section 3 is proposed. This approximation, which is shown to converge uniformly to the true copula, is then applied to exploit some characteristics of the copulas related to Manneville-Pomeau process through statistical and graphical analysis. Some computational and random variate generation problems are also addressed. In Section 6 we illustrate the usefulness of the theory by deriving a fast procedure to estimate the underlying parameter in Manneville-Pomeau processes. Conclusions are reserved to Section 7.
\section{Some Background}
\setcounter{equation}{0}
\indent In this section we shall briefly review some basic results on Manneville-Pomeau transformations and related processes  as well as some concepts on copulas needed later. We start with the definition of the Manneville-Pomeau transformation.

        \begin{df}\label{mpmap}
        The map $T_s:[0,1]\longrightarrow[0,1]$, given by
        \small
        \[T_s(x)=x+x^{1+s}(\md1),\]
        \normalsize
        for $s>0$, is called the \emph{Manneville-Pomeau transformation} (MP \emph{transformation}, for short).
        \end{df}

In what follows, $\lambda$ shall denote the Lebesgue measure in $I:=[0,1]$ and the $k$-fold composition will be denoted, as usual, by $T_s^k=T_s\circ\cdots\circ T_s$. Figure 1 shows the plot of the MP transformation for the values of $s\in\{  0.5, 1, 10, 100\}$. The plots show the usual behavior of  the MP transformations: for any $s$, they are  increasing and differentiable functions by parts in $I$. Furthermore, for any $s>0$, the function $T_s^k$ will have exactly $2^k$ parts.
\begin{figure}[!ht]\label{1}
\centering
\mbox{
\includegraphics[width=0.22\textwidth]{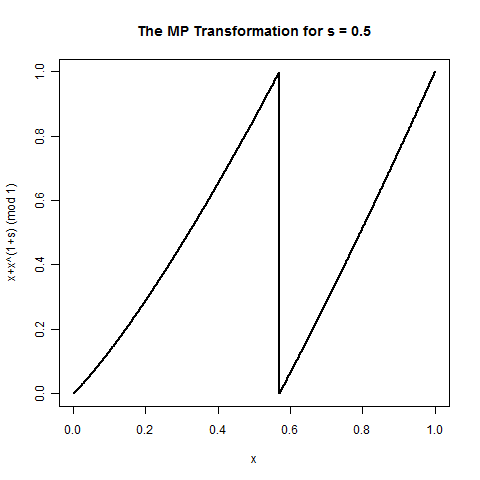}\hspace{-.1cm}
\includegraphics[width=0.22\textwidth]{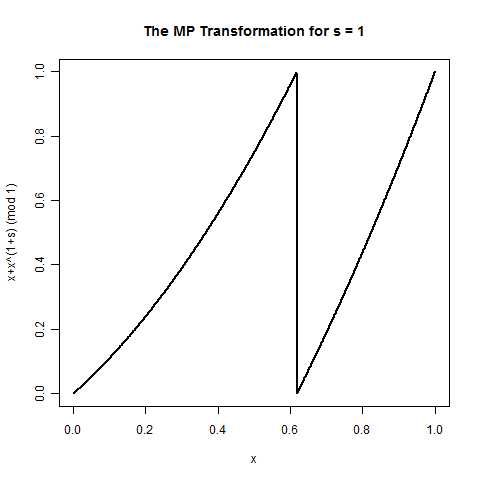}\hspace{-.1cm}
\includegraphics[width=0.22\textwidth]{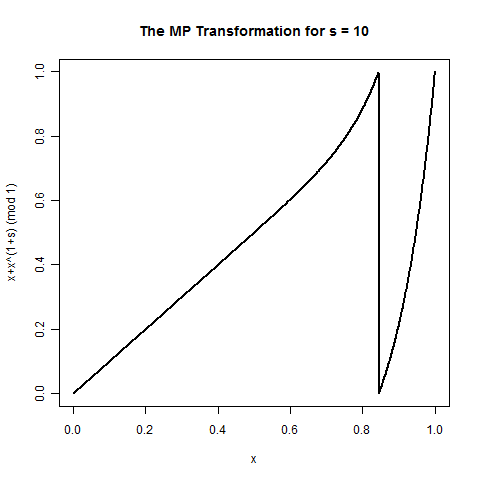}\hspace{-.1cm}
\includegraphics[width=0.22\textwidth]{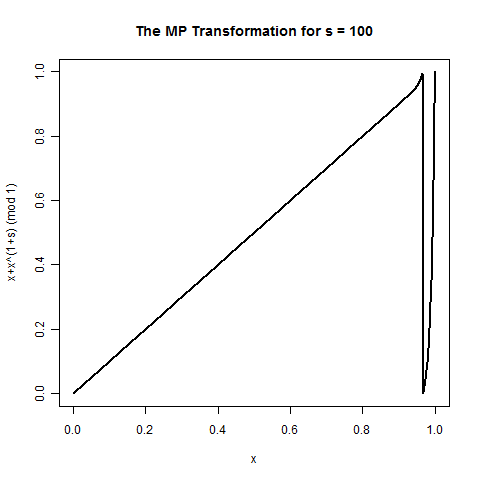}
  }
  \caption{Plot of the Manneville-Pomeau transformation for different values of $s\in\{0.5,1,10,100\}$. }\label{f5}
\end{figure}

Pianigiani (1980) shows the existence of a $T_s$-invariant and absolutely continuous measure with respect to the Lebesgue measure in $I$ which will be denoted henceforth by $\mu_s$. However, the proof uses Perron-Frobenius operator theory and is, for practical purposes, non-constructive so that an explicit form for a $T_s$-invariant measure is unknown. However, this measure will  be a Sinai-Bowen-Ruelle (SBR) measure in the sense that the weak convergence
\small
\begin{equation}\label{sbr}
\frac{1}{n}\sum_{k=0}^{n-1}\delta_{T_s^k(x)}(A)\longrightarrow \mu_s(A)
\end{equation}
\normalsize
holds for almost all $x\in I$ and all $\mu_s$-continuity sets\footnote{Recall that a set $A$ is a $\mu$-continuity set if $\mu(\partial A)=0$, where $\partial A$ denotes the boundary of $A$. The measure theoretical results applied here can be found, for instance, in Royden (1988). A good reference in weak convergence of probability measures is Billingsley (1999) and for ergodic theoretical related results, see Pollicott and Yuri (1998). } $A$, where $\delta_{a}(\cdot)$ is the Dirac measure at $a$.

As a dynamical system, the triple $(I,\mu_s,T_s)$ is exact (that is, $\lim_{k\rightarrow\infty}(\mu_s\circ T_s^k)(A)=1$, for all positive $\mu_s$-measurable sets $A$) which implies ergodicity and strong-mixing. When $s<1$, $\mu_s$ is a probability measure, while if $s\geq 1$, $\mu_s$ is no longer finite, but $\sigma$-finite (see Fisher and Lopes (2001)). Furthermore, it can be shown that $\mu_s$ has a positive, bounded continuous Radon-Nikodym derivative $\mathrm{d} \mu_s=h_s(x)\mathrm{d}x$, fact that will be useful later. For further details in the theory of MP transformations and related results, we refer to Pianigiani (1980), Young (1999), Maes et al. (2000) and Fisher and Lopes (2001). For applications, see Zebrowsky (2001), Olbermann et al. (2007) and Lopes and Lopes (1998).

        \begin{df}\label{mpp}
         Let $s\in(0,1)$ and let $U_0$ be a random variable  distributed according to (the probability measure) $\mu_s$.  Let $\phi:[0,1]\longrightarrow\R$ be a function in $\mathcal{L}^1(\mu_s)$. The stochastic process given by
         \small
          \[X_t=(\phi\circ T_s^t)(U_0),  \ \ \mbox{for all}\ \ t\in\N,\]
          \normalsize
        is called a \emph{Manneville-Pomeau process} (or MP \emph{process}, for short).
        \end{df}

The MP process, as defined above, is stationary since $\mu_s$ is a $T_s$-invariant measure and $\mu_s\ll\lambda$. It is also ergodic since $\mu_s$ is ergodic for $T_s$. 
By its turn, copulas are distribution functions whose marginals are uniformly distributed on $I$. The copula literature has grown enormously in the last decade, especially in terms of empirical applications and have become standard tools in financial data analysis (see Nelsen (2006) and references therein).  The next theorem, known as Sklar's theorem, is the key result for copulas and elucidates the role played by them. See Schweizer and Sklar (2005) for a proof.

        \begin{thm}[Sklar]Let $X_1,\cdots,X_n$ be random variables with mar\-gi\-nals $F_1,\cdots,F_n$, respectively, and joint distribution function $H$. Then, there exists a copula $C$ such that,
        \[H(x_1,\cdots,x_n)=C\big(F_1(x_1),\cdots,F_n(x_n)\big), \ \ \ \mbox{for all}\ \  (x_1,\cdots,x_n) \in \R^n.\]
        If the $F_i$'s are continuous, then $C$ is unique. Otherwise, $C$ is uniquely determined on
        $\mathrm{Ran}(F_1)\times\cdots\times\mathrm{Ran}(F_n)$. The converse also holds. Furthermore,
        \[C(u_1,\cdots,u_n)=H\big(F_1^{(-1)}(u_1),\cdots,F_n^{(-1)}(u_n)\big),\ \ \  \mbox{for all} \ \  (u_1,\cdots,u_n)\in I^n,\]
        where for a function $F$, $F^{(-1)}$ denotes its pseudo-inverse given by $F^{(-1)}(x):=\inf\big\{u\in \mathrm{Ran}(F): F(u)\geq x\big\}.$
        \end{thm}

The next theorem, whose proof can be found, for instance, throughout Nelsen (2006), shall prove very useful in what follows. Except stated otherwise, the measure implicit to phrases like ``almost sure'', ``almost everywhere'' and so on will be the (appropriate) Lebesgue measure.

        \begin{thm}\label{ver}
        Let $X$ and $Y$ be continuous random variables with copula $C$. If $f$ is an almost everywhere decreasing function then $C_{f(X),Y}(u,v)=u-C_{X,Y}(u,1-v)$. Furthermore, if $f_1$ and $f_2$ are functions increasing almost everywhere, then $C_{f_1(X),f_2(Y)}(u,v)=C_{X,Y}(u,v)$.
        \end{thm}

For an introduction to copulas, we refer the reader to Nelsen (2006). For more details and extensions to the multivariate case with emphasis in modeling and dependence concepts, see Joe (1997). The theory of copulas is also intimately related to the theory of probabilistic metric spaces, see Schweizer and Sklar (2005) for more details in this matter.

\section{ Copulas and MP Processes: Bidimensional Case}
\setcounter{equation}{0}

In this section we shall investigate the bidimensional copulas associated to pairs of random variables coming from MP processes which we shall call MP \emph{copulas}.  As we will see later, the multidimensional case is very similar to the bidimensional case, so we shall give special attention to the latter.

First, let $\{X_n\}_{n\in\N}$ be an MP process with parameter $s\in(0,1)$ and $\phi\in\L^1(\mu_s)$ be an increasing almost everywhere function. Throughout this section and in the rest of the paper, we shall treat $s\in(0,1)$ as a given fixed number. Let
\small
\[F_0(x):=\P(U_0\leq x)=\mu_s\big([0,x]\big).\]
\normalsize
Since $\mu_s\ll\lambda$, $\mu_s$ is non-atomic and, therefore, $F_0$ will be (uniformly) continuous. The existence of a positive Radon-Nikodym density for $\mu_s$ also shows that $F_0$ will be increasing and its inverse will be well defined.  Let $F_t$ be the distribution function of $T^t_s(U_0)$, for all $t\in\N$.  For $x\in I$, notice that
\small
\begin{equation}\label{ft}
F_t(x):=\P\big(T_s^t( U_0)\leq x\big)=\mu_s\big((T_s^t)^{-1}\big([0,x]\big)\big)=\mu_s\big([0,x]\big)=F_0(x),
\end{equation}
\normalsize
since $\mu_s$ is a $T_s$-invariant measure.

In what follows, we shall need the solution for the inequality $T^t_s(X)\leq y$, $y\in(0,1)$, in $X$, for $X$ a random variable taking values in $I$. Now, since each of the $2^t$ parts of $T_s^t$ is one-to-one in its domain, the inverse of $T_s^t$ will also be continuous by parts and each part will also be a one-to-one function in its domain. Let $0=a_{t,0},\cdots,a_{t,2^t}=1$ be the end points of each part of $T_s^t$. We shall call each interval $[a_{t,k},a_{t,k+1})$ a \emph{node} of $T_s^t$, for $k=0,\cdots,2^t-1$ and $t>0$. The (piecewise) inverse of $T_s^t$ can be conveniently written as
\small
\begin{eqnarray}\label{ss}
(T^t_s)^{-1}:I&\longrightarrow& I^{2^t}\nonumber\\
y&\longmapsto&\big(\T_{t,0}(y),\cdots,\T_{t,2^t-1}(y)\big),
\end{eqnarray}
\normalsize
where $\T_{t,k}(y)$ denotes the inverse of $T_s^t$ restricted to its $k$-th node, for all $k\in\{0,\cdots,2^t-1\}$.
Notice that both $\mathcal{T}_{t,k}$ and $a_{t,k}$ depend on $s$ for each $k$, but since no confusion will arise, and for the sake of simplicity, we shall omit this dependence from the notation as we shall do in several other occasions. Now, the solution of the inequality $T_s^t(X)\leq y$ in $X$ can be determined and is given by $X\in A_{t,0}(y)\,\cup\cdots\cup A_{t,2^t-1}(y)$, where
\small
\begin{equation}\label{ais}
A_{t,k}(y)=\big[a_{t,k},\mathcal{T}_{t,k}(y)\big],
\end{equation}
\normalsize
which will be a proper closed subinterval of $[a_{t,k},a_{t,k+1})$, for each $k=0,\cdots, 2^t-1$. Notice that $A_{t,k}(y)$ (whose dependence on $s$ was omitted from the notation) is just the inverse image of $[0,y]$ by the transformation $T^t_s$ restricted to the node $[a_{t,k},a_{t,k+1})$. We can now use this result to prove the following useful lemma.

    \begin{lema}\label{aea}
    Let   $X$ be a random variable taking values in $I$ and let $T_s$ be the \emph{MP} transformation with parameter $s>0$. Then, for any $t\in\N$ and $x\in I$,
    \small
    \[\P\big(T^t_s(X)\leq x\big) = \P\big(X\in \cup_{k=0}^{2^t-1}A_{t,k}(x)\big)=\sum_{k=0}^{2^t-1}\P\big(X\in A_{t,k}(x)\big),\]
    \normalsize
    where $A_{t,k}$'s are given by \eqref{ais}.
    \end{lema}

\begin{proof}
The result follows easily from what was just discussed and from the fact that the intervals $A_{t,k}$'s are (pairwise) disjoints.\fim
\end{proof}

As for the copulas related to MP processes, in view of the stationarity of the MP process, the following result follows easily.

        \begin{prop}\label{invar}
        Let $\{X_n\}_{n\in\N}$ be an \emph{MP} process with parameter $s\in(0,1)$ and $\phi\in\L^1(\mu_s)$ be an almost everywhere increasing function. Then, for any $t, h\in\N$,
        \[C_{X_t,X_{t+h}}(u,v)=C_{X_0,X_h}(u,v),\]
        everywhere in $ I^2$.
        \end{prop}

\begin{proof}
As consequence of the stationarity of $\{X_t\}_{t\in\N}$, if we let the joint distribution of the pair $(X_p,X_q)$ for any $p,q\in\N$, $p\neq q$, be denoted by $\widetilde{H}_{p,q}(\cdot,\cdot)$, it follows that for all $x,y\in (0,1)$, $t\in\N$ and $h\in\N^\ast:=\N\backslash\{0\}$, $\widetilde{H}_{t,t+h}(x,y)=\widetilde{H}_{0,h}(x,y)$. Now, upon applying Sklar's Theorem and \eqref{ft}, it follows that
\small
\[C_{X_t,X_{t+h}}(u,v)\!=\!\widetilde{H}_{t,t+h}\big(F_t^{-1}(u),F_{t+h}^{-1}(v)\big)\!=\!\widetilde{H}_{0,h}\big(F_0^{-1}(u),F_{h}^{-1}(v)\big) \!=\! C_{X_0,X_h}(u,v),\]
\normalsize
for all $(u,v)\in I^2$.\fim
\end{proof}

        \begin{cor}\label{coraux}
        Let $T_s$ be the \emph{MP} transformation for some $s\in(0,1)$, $\mu_s$ be a $T_s$-invariant probability measure and let $U_0$ be distributed as $\mu_s$. Then, for any $t,h\in\N$, $h\neq0$,
        \[C_{T^t_s(U_0),T^{t+h}_s(U_0)}(u,v)=C_{U_0,T^h_s(U_0)}(u,v)\]
        everywhere in $I^2$.
        \end{cor}

\begin{proof}
Immediate from Theorem \ref{ver} applied to Proposition \ref{invar}.\fim
\end{proof}
Now we turn our attention to determine the copula associated to any pair of random variables $(X_p,X_q)$, $p,q\in\N$, obtained from an MP process with $\p$ increasing almost everywhere. For the sake of simplicity, let us introduce the following functions: let $h$ be a positive integer and for $k=0,\cdots,2^h-1$, let \small$\mathcal{F}_{h,k}:I\rightarrow\big[F_0(a_{h,k}),F_0(a_{h,k+1})\big]$\normalsize be given by
\small
\[\mathcal{F}_{h,k}(x):=F_0\big(\T_{h,k}\big( F^{-1}_0(x)\big)\big).\]
\normalsize
Notice that for each $k$, $\mathcal{F}_{h,k}(0)=F_0(a_{h,k})$ and $\mathcal{F}_{h,k}(1)=F_0(a_{h,k+1})$ and $\mathcal{F}_{h,k}$ is a one to one, increasing and uniformly continuous function.

        \begin{prop}\label{copcr}
        Let $\{X_n\}_{n\in\N}$ be an \emph{MP} process with parameter $s\in(0,1)$, $\phi\in\L^1(\mu_s)$ be an increasing almost everywhere function and let $F_0$ be the distribution function of $U_0$. Then, for any $t,h\in\N$, $h\neq 0$ and $(u,v)\in I^2$,
        \small
        \begin{equation}\label{cop}
        C_{X_t,X_{t+h}}(u,v)=\left(\sum_{k=0}^{n_0-1}\mathcal{F}_{h,k}(v)-F_0(a_{h,k})\!\right)\!\delta_{\N^\ast}\!(n_0)+\min\!\big\{u,\mathcal{F}_{h,n_0}(v)\!\big\} -F_0(a_{h,n_0}),
        \end{equation}
        \normalsize
        where $\delta_{\N^\ast}\!(x)$ equals 1, if $x\in\N^\ast$ and 0, otherwise, $\{a_{h,k}\}_{k=0}^{2^h}$ are the end points of the nodes of $T^h_s$ and $n_0:=n_0(u;h)=\big\{k:u\in\big[F_0(a_{h,k}),F_0(a_{h,k+1})\big)\big\}\in\{0,\cdots,2^h-1\}$.
        \end{prop}
\begin{proof}
By Propositions \ref{invar} and \ref{ver}, it suffices to derive the copula of the pair  $\big(U_0,T_s^h(U_0)\big)$. So let again $\{X_n\}_{n\in\N}$ be an MP process with parameter $s\in(0,1)$ and $\phi\in\L^1(\mu_s)$ be an increasing almost everywhere function and let $H_{0,h}(\cdot,\cdot)$ denote the distribution function of the pair $\big(U_0,T^h_s(U_0)\big)$. Notice that
\small
 \begin{align*}
H_{0,h}(x,y)=&\P(U_0\leq x, T^h_s(U_0)\leq y)=\P\big(U_0\leq x, U_0\in\cup_{k=0}^{2^h-1}A_{h,k}(y)\big)\nonumber\\
&=\P\big(U_0\in[0,x]\cap\cup_{k=0}^{2^h-1}A_{h,k}(y)\big)=\P\big(U_0\in\cup_{k=0}^{2^h-1}\big[[0,x]\cap A_{h,k}(y)\big]\big)\\
&=\sum_{k=0}^{2^h-1}\P\big(U_0\in[0,x]\cap A_{h,k}(y)\big),
\end{align*}
\normalsize
for any $x,y\in(0,1)$. Now let $n_1:=n_1(x;h)=\big\{k:x\in[a_{h,k},a_{h,k+1})\big\}\in\{0,\cdots, 2^h-1\}$ and assume for the moment that $n_1\geq 1$. Since $A_{h,k}(y)=\big[a_{h,k},\T_{h,k}(y)\big]$, it follows
\small
\begin{align*}
H_{0,h}(x,y)&=\sum_{k=0}^{n_1-1}\P\big(U_0\in A_{h,k}(y)\big)+\P\big(U_0\in A_{h,n_1}(y)\cap[a_{h,n_1},x]\big)\\
&=\sum_{k=0}^{n_1-1} \mu_s\big(A_{h,k}(y)\big)+\mu_s\big(\big[a_{h,n_1},\T_{h,n_1}(y)\big]\cap[a_{h,n_1},x]\big)\\
&=\sum_{k=0}^{n_1-1} \mu_s\big(\big[a_{h,k},\T_{h,k}(y)\big]\big)+\mu_s\big(\big[a_{h,n_1}, \min\{x,\T_{h,n_1}(y)\}\big]\big),
\end{align*}
\normalsize
which can be written, since $F_0(x)=\mu_s([0,x])$ is increasing, as
\small
\[H_{0,h}(x,y)=\sum_{k=0}^{n_1-1}\left[F_0\big(\T_{h,k}(y)\big)-F_0(a_{h,k})\right] +\min\big\{F_0(x),F_0\big(\T_{h,n_1}(y)\big)\big\}-F_0(a_{h,n_1}).\]
\normalsize
If $n_1=0$, the summation is absent of the formula and we have
\small
\[H_{0,h}(x,y)=\min\big\{F_0(x),F_0\big(\T_{h,0}(y)\big)\big\}-F_0(a_{h,0}),\]
\normalsize
so that, in any case, we have
\small
\[H_{0,h}(x,y)=\left(\sum_{k=0}^{n_1-1}\left[F_0\big(\T_{h,k}(y)\big)-F_0(a_{h,k})\right]\! \right)\!\delta_{\N^\ast}\!(n_1) + \min\big\{F_0(x),F_0\big(\T_{h,n_1}(y)\big)\big\} -F_0(a_{h,n_1}).\]
\normalsize
Now upon applying Sklar's Theorem, it follows that
\small
\begin{align*}
C_{U_0,T_s^h(U_0)}(u,v)&=H_{0,h}\big(F^{-1}_0(u),F^{-1}_h(v)\big)=H_{0,h}\big(F^{-1}_0(u),F^{-1}_0(v)\big)\\
&=\left(\sum_{k=0}^{n_0-1}\mathcal{F}_{h,k}(v)-F_0(a_{h,k})\!\right)\delta_{\N^\ast}(n_0)+ \min\big\{u,\mathcal{F}_{h,n_0}(v)\big\} -F_0(a_{h,n_0}),
\end{align*}
\normalsize
where $n_0:=n_0(u;h)=n_1\big(F_0^{-1}(u);h\big)=\big\{k:u\in\big[F_0(a_{h,k}),F_0(a_{h,k+1})\big)\big\}$. The result now follows from Proposition \ref{invar}.\fim
\end{proof}

        \begin{rmk}
        Notice that the copula \eqref{cop} can be expressed in terms of $\mu_s$ as
        \small
        \begin{eqnarray}\label{copmu}
        C_{X_t,X_{t+h}}(u,v)\!&=&\!\!\left(\sum_{k=0}^{n_0-1}\mu_s\Big(\big[a_{h,k},\T_{h,k}\big( F^{-1}_0(v)\big)\big]\Big)\!\right)\!\delta_{\N^\ast}\!(n_0) +\hphantom{60pt}\nonumber\\
        &&\hspace{0.9cm}+\,\mu_s\Big(\big[a_{h,n_0},\min\big\{F_0^{-1}(u),\T_{h,n_0}\big( F^{-1}_0(v)\big)\big\}\big]\Big),
        \end{eqnarray}
        \normalsize
        which will prove useful in Section 5. Also, expression \eqref{copmu} is helpful if one desires to verify directly that the marginals of \eqref{cop} are indeed uniform.
        \end{rmk}

\indent In the next proposition we address the case where $\p$ is an almost everywhere decreasing function. In view of Theorem \ref{ver}, one could, at first glance, think that a result like $C_{X_0,X_h}=C_{X_t,X_{t+h}}$ would not hold anymore, but in fact it still does, as it is shown in the next proposition.

        \begin{prop}\label{copdec}
        Let $\{X_n\}_{n\in\N}$ be an \emph{MP} process with parameter $s\in(0,1)$, $\phi\in\L^1(\mu_s)$ be an almost everywhere decreasing function and let $F_0$ be the distribution function of $U_0$. Then, $C_{X_0,X_h}(u,v)=C_{X_t,X_{t+h}}(u,v)$ everywhere in $I^2$ and, for any $t,h\in\N$ and $h\neq0$, \small
        \begin{eqnarray}\label{copb}
        C_{X_t,X_{t+h}}(u,v)&=&u+v-1+\left(\sum_{k=0}^{n_0}\left[\mathcal{F}_{h,k}(1-v)-F_0(a_{h,k})\right]\right)\!\delta_{\N^\ast}\!(n_0) +\hphantom{60pt}\nonumber\\
       &&\hspace{1cm}+\,\min\big\{1-u,\mathcal{F}_{h,n_0}(1-v)\big\}-F_0(a_{h,n_0}),
        \end{eqnarray}
        \normalsize
        for all $(u,v)\in I^2$, where $\{a_{h,k}\}_{k=0}^{2^h}$ are the end points of the nodes of $T^h_s$ and $n_0:=n_0(u;h)=\big\{k:u\in\big(1-F_0(a_{h,k+1}),1-F_0(a_{h,k})\big]\big\}$.
        \end{prop}

\begin{proof} Since the inverse of an almost everywhere decreasing function is still decreasing almost everywhere and $X_t=\p\big(T_s^t(U_0)\big)$, upon applying  Theorem \ref{ver} twice, it follows that
\small
\begin{eqnarray*}
C_{T^t_s(U_0),\,T^{t+h}_s(U_0)}(u,v)&=&C_{\ip(X_t),\,\ip(X_{t+h})}(u,v)=u-C_{X_t,\ip(X_{t+h})}(u,1-v)\\
&=&u-\big(1-v-C_{X_t,\,X_{t+h}}(1-u,1-v)\big),
\end{eqnarray*}
\normalsize
or, equivalently (changing $u$ by $1-u$ and $v$ by $1-v$),
\small
\begin{equation}\label{mm}
C_{X_t,\,X_{t+h}}(u,v)=u+v-1+C_{T^t_s(U_0),\,T^{t+h}_s(U_0)}(1-u,1-v).
\end{equation}
\normalsize
Now \eqref{copb} follows upon applying Proposition \ref{copcr} with the identity map and substituting equation \eqref{cop} into \eqref{mm}. As for the equality $C_{X_0,X_h}(u,v)=C_{X_t,X_{t+h}}(u,v)$, Corollary \ref{coraux} and Theorem \ref{ver} applied to \eqref{mm} yield
\small
\begin{eqnarray*}
C_{X_t,\,X_{t+h}}(u,v)&=&u+v-1+C_{U_0,\,T^{h}_s(U_0)}(1-u,1-v)\\
&=&u+v-1+C_{\ip(\p(U_0)),\ip(\p(T^{h}_s(U_0)))}(1-u,1-v)\\
&=&C_{\p(U_0),\p(T^{h}_s(U_0))}(u,v)\,\,=\,\,C_{X_0,X_h}(u,v),
\end{eqnarray*}
\normalsize
everywhere in $I^2$, as desired.\fim
\end{proof}

\begin{rmk}
In view of the ``stationarity'' results of Theorems \ref{invar} and \ref{copdec}, a copula associated to a pair $(X_t,X_{t+h})$ from an MP process will be referred as \emph{lag $h$} MP \emph{copula}.
\end{rmk}

The copulas in \eqref{cop} and \eqref{copb} are both singular, as it can be readily verified. So the question that naturally arises is, for each $h$, what is the support of $C_{X_t,X_{t+h}}$? The question is addressed in the next proposition, which will be useful in Sections 5 and 6. For simplicity, for a given MP process and $h>0$, let $\ell_{h,k}^+,\ell_{h,k}^-:\big[F_0(a_{h,k}),F_0(a_{h,k+1})\big)\rightarrow I$ be functions defined by
\[\ell_{h,k}^+(x)=\frac{x-F_0(a_{h,k})}{F_0(a_{h,k+1})-F_0(a_{h,k})}\quad\mbox{ and }\quad\ell_{h,k}^-(x)=\frac{F_0(a_{h,k+1})-x}{F_0(a_{h,k+1})-F_0(a_{h,k})},\]
for all $k=0,\cdots,2^h-1$. Notice that, for each $k$, $\ell_{h,k}^+$ is the linear function connecting the points $\big(F_0(a_{h,k}),0\big)$ and $\big(F_0(a_{h,k+1}),1\big)$, while $\ell_{h,k}^-$ connects the points $\big(F_0(a_{h,k}),1\big)$ and $\big(F_0(a_{h,k+1}),0\big)$.

        \begin{prop}\label{supp}
        Let $\{X_n\}_{n\in\N}$ be an \emph{MP} process with parameter $s\in(0,1)$, for $\phi_1\in\L^1(\mu_s)$ an almost everywhere increasing function and let $\{Y_n\}_{n\in\N}$ be an MP process with parameter $s\in(0,1)$, for $\phi_2\in\L^1(\mu_s)$ an almost everywhere decreasing function. Also let  $F_0$ be the distribution function of $U_0$. Then, for any $t,h\in\N$, $h>0$,
        \small
        \begin{equation}\label{spt}
        \mathrm{supp}\{C_{X_t,X_{t+h}}\}=\cup_{k=0}^{2^h-1}\big\{\big(u,\ell_{h,k}^+(u)\big): u\in\big[F_0(a_{h,k}),F_0(a_{h,k+1})\big) \big\}
        \end{equation}
        and
        \small
        \begin{equation}\label{sptdec}
        \mathrm{supp}\{C_{Y_t,Y_{t+h}}\}=\cup_{k=0}^{2^h-1}\big\{\big(u,\ell_{h,k}^-(u)\big): u\in\big[F_0(a_{h,k}),F_0(a_{h,k+1})\big) \big\}.
        \end{equation}
        \normalsize
        \end{prop}
\begin{proof}
Let $R=[u_1,u_2]\times[v_1,v_2]$ be a rectangle in $I^2$ and let its $C_{X_t,X_{t+h}}$-volume be denoted by $V_{C_{\boldsymbol X}}(R)$. Let $k\in\{0,\cdots,2^h-1\}$ be fixed and suppose that $u_i\in\big[F_0(a_{h,k}),F_0(a_{h,k+1})\big]$. This implies that $n_0=k$ for all four terms in $V_{C_{\boldsymbol X}}(R)$, hence the summands and constants on the copula cancel out so that we have
\small
\begin{align*}
V_{C_{\boldsymbol X}}(R)&=\min\big\{u_1, \mathcal{F}_{h,k}(v_1)\big\}+\min\big\{u_2, \mathcal{F}_{h,k}(v_2)\big\}-\min\big\{u_1, \mathcal{F}_{h,k}(v_2)\big\}-\min\big\{u_2, \mathcal{F}_{h,k}(v_1)\big\}\\
&=V_M\big([u_1,u_2]\times[\mathcal{F}_{h,k}(v_1),\mathcal{F}_{h,k}(v_1)]\big),
\end{align*}
\normalsize
where $M(u,v)=\min\{u,v\}$ is the Frech\`et upper bound copula whose support is the main diagonal in $I^2$. Since $[u_1,u_2]\times[\mathcal{F}_{h,k}(v_1),\mathcal{F}_{h,k}(v_1)]\subset[F_0(a_{h,k}),F_0(a_{h,k+1})]^2$, $V_{C_{\boldsymbol X}}(R)>0$ if, and only if, $R\,\cap\big\{\big(u,\ell_{h,k}^+(u)\big): u\in\big[F_0(a_{h,k}),F_0(a_{h,k+1})\big) \big\}\neq \emptyset$.

Analogously, denoting the $C_{Y_t,Y_{t+h}}$-volume of $R$ by $V_{C_{\boldsymbol Y}}(R)$, if $u_i\in\big[1-F_0(a_{h,k}),1-F_0(a_{h,k+1})\big]$, we have
\small
\begin{align}\label{sptaux}
V_{C_{\boldsymbol Y}}(R)&=\min\big\{1-u_1, \mathcal{F}_{h,k}(1-v_1)\big\}+\min\big\{1-u_2, \mathcal{F}_{h,k}(1-v_2)\big\}-\nonumber\\
&\hskip2cm-\ \min\big\{1-u_1, \mathcal{F}_{h,k}(1-v_2)\big\}-\min\big\{1-u_2, \mathcal{F}_{h,k}(1-v_1)\big\}\nonumber\\
&=V_M\big([1-u_1,1-u_2]\times[\mathcal{F}_{h,k}(1-v_2),\mathcal{F}_{h,k}(1-v_1)]\big).
\end{align}
\normalsize
Since $[1-u_1,1-u_2]\times[\mathcal{F}_{h,k}(1-v_1),\mathcal{F}_{h,k}(1-v_2)]\subset[F_0(a_{h,k}),F_0(a_{h,k+1})]^2$, $V_{C_{\boldsymbol Y}}(R)$
is positive if, and only if, $R\,\cap\big\{\big(u,\ell_{h,k}^-(u)\big): u\in\big[F_0(a_{h,k}),F_0(a_{h,k+1})\big) \big\}\neq \emptyset$ (notice the terms $1-v_i$ in expression \eqref{sptaux}, for $i=1,2$).
Now \eqref{spt} and \eqref{sptdec} follow by observing that $I=\bigcup_{k=0}^{2^h-1}\big[F_0(a_{h,k}),F_0(a_{h,k+1})\big]=\bigcup_{k=0}^{2^h-1}\big[1-F_0(a_{h,k+1}),1-F_0(a_{h,k})\big]$.\fim
\end{proof}

\begin{rmk}
We end up this section by noticing that as an application of Propositions \ref{invar} and \ref{copdec}, together with the so-called copula version of Hoeffding's lemma (see Nelsen (2006)), we can show in a rather different way that an MP process is weakly stationary. Let $F_{X_t}$ be the distribution function of $X_t$ and notice that $F_{X_t}(x)=F_{X_0}(x)$, for all $t\in\N$, by the stationarity of $\{X_t\}_{t\in \N}$ and since
$C_{X_t,X_{t+h}}(u,v)=C_{X_0,X_{h}}(u,v)$, the result follows immediately.
\end{rmk}

\section{Multidimensional Case}
\setcounter{equation}{0}

\indent In this section we are interested in extending the results from the previous section to the multidimensional case, that is, in this section we are interested in deriving the copulas associated to $n$-dimensional vectors $(X_{t_1},\cdots,X_{t_n})$, $t_1,\cdots,t_n\in\N$, coming from an MP process with $\p$ an increasing almost everywhere function. In view of Theorem \ref{ver}, it suffices to derive the copula associated to the vector $\big(T^{t_1}_s(U_0),\cdots,T^{t_n}_s(U_0)\big)$. It turns out that there are more similarities between the bidimensional and multidimensional cases than one could expect. In fact, an expression very similar in form to \eqref{cop} holds for the multidimensional case as well.

Let $\{X_n\}_{n\in\N}$ be an MP process with parameter $s\in(0,1)$ and $\phi\in\L^1(\mu_s)$ be an almost everywhere increasing function. For the sake of simplicity, we shall use the following notation: let $a,b\in\N$, $a<b$, we shall write $x_{a:\,b}:=(x_a,\cdots,x_b)$ and for a function $f$, $f(x_{a:\,b}):=\big(f(x_a),\cdots,f(x_b)\big)$. Again we shall denote the distribution function of $U_0$ by $F_0$.

        \begin{thm}\label{mdim}
        Let $\{X_n\}_{n\in\N}$ be an \emph{MP} process with parameter $s\in(0,1)$, with $\phi\in\L^1(\mu_s)$ an almost everywhere increasing function. Let $t_1,\cdots,t_n\in\N$ and set $h_i:=t_i-t_1$. Then, for all $(u_1,\cdots,u_n)\in I^n$,
        \small
        \begin{align}\label{mcop}
        C_{X_{t_1},\cdots,X_{t_n}}(u_1,\cdots,u_n)&=\left(\sum_{k=0}^{n_0-1} F_0\Big(b_{h_n,k}\big(F^{-1}_0(u_{2:n})\big)\Big)-F_0(a_{h_n,k})\!\right)\! \delta_{\N^\ast}\!(n_0)  + \nonumber\\
         &\hskip1.5cm+\,\min\big\{u_1,F_0\big(b_{h_n,n_0}\big(F_0^{-1}(u_{2:n})\big)\big)\big\}-F_0(a_{h_n,n_0}),
        \end{align}
        \normalsize
        where  $n_0:=n_0\big(u_1,n)=\big\{k:u_1\in\big[F_0(a_{h_n,k}),F_0(a_{h_n,k+1})\big)\big\}$,  $\{a_{h_n,k}\}_{k=0}^{2^h}$ are the end points of the nodes of $T^{h_n}_s$, for $i=2,\cdots,n$, $j=0,\cdots 2^{h_i}-1$,  $\T_{h_i,j}$ is given by \eqref{ss} and for a vector $(x_2,\cdots,x_n)\in I^{n-1}$,
        $b_{h_n,k}(x_{2:n})=\displaystyle{\min_{i=2,\cdots,n}}\big\{c_{i}(x_i;h_n,k)\big\}$,
        with
        \small
        \[c_i(x_i;h_n,k)=\left\{
        \begin{array}{cc}
        a_{h_n,k}, & \mbox{ if } \ \  B_i(x_i;h_n,k)=\emptyset; \\
        B_i(x_i;h_n,k), & \mbox{otherwise}.
        \end{array}\right.\]
        \normalsize
         and
         \small
         \[ B_i(x_i;h_n,k)=\displaystyle{\min_{j=0,\cdots,2^{h_i}-1}}\hspace{-.1cm}\big\{\T_{h_i,j}(x_i):\T_{h_i,j}(x_i)>a_{h_n,k} \mbox{ and }a_{h_i,j}<a_{h_n,k+1} \big\}.\]
         \normalsize
        \end{thm}
\begin{proof}
Let $\{X_n\}_{n\in\N}$ be an MP process with parameter $s\in(0,1)$ and $\phi\in\L^1(\mu_s)$ be an almost everywhere increasing function. Without loss of generality, we can assume that  $0<t_1<\cdots<t_n$. In view of Theorem \ref{ver}, it suffices to work with the vector $\big(T^{t_1}_s(U_0),\cdots,T^{t_n}_s(U_0)\big)$. Let $H_{t_1,\cdots,t_n}$ be the distribution function of $\big(T^{t_1}_s(U_0),\cdots,T^{t_n}_s(U_0)\big)$. Let  $h_i=t_i-t_1$, for each $i =1,\cdots,n$, and notice that $h_i>0$ since $t_1<t_i$, for all $i=2,\cdots,n$. Let $(x_1,\cdots,x_n)\in (0,1)^n$ and for the sake of simplicity, let $Y_{t_1}:=T^{t_1}_s(U_0)$, so that we have
\small
\begin{align}\label{aux2}
H_{t_1,\cdots,t_n}(x_1,\cdots,x_n)&=\P\big(T^{t_1}_s(U_0)\leq x_1,\cdots,T^{t_n}_s(U_0)\leq x_n\big)\nonumber\\
&=\P\big(Y_{t_1}\leq x_1,T^{h_2}_s(Y_{t_1})\leq x_2,\cdots,T^{h_n}_s(Y_{t_1})\leq x_n\big)\nonumber\\
&=\P\Big(Y_{t_1}\!\!\in[0,x_1],Y_{t_1}\!\!\in\cup_{k=0}^{2^{h_2}-1}A_{h_2,k}(x_2),\cdots,Y_{t_1}\!\!\in\cup_{k=0}^{2^{h_n}-1}A_{h_n,k}(x_n)\!\Big)\nonumber\\
&=\P\Big(Y_{t_1}\!\!\in[0,x_1]\cap_{i=2}^n\big[\cup_{k=0}^{2^{h_i}-1}A_{h_i,k}(x_i)\big]\Big)\nonumber\\
&=\P\Big(U_0\in \cap_{i=2}^n\cup_{k=0}^{2^{h_i}-1}\big[[0,x_1]\cap A_{h_i,k}(x_i)\big]\Big),
\end{align}
\normalsize
where $A_{h_i,k}$'s are given by \eqref{ais} and the last equality is a consequence of the $T_s$-invariance of $\mu_s$. For $k=0,\cdots,2^{h_n-1}$, let
\small
\[\widetilde{A}_{h_n,k}(x_{2:n})=A_{h_n,k}(x_n)\cap_{i=2}^{n-1}\big[\cup_{j=0}^{2^{h_i}-1}A_{h_i,j}(x_i)\big].\]
\normalsize
In order to simplify the notation, for $i=2,\cdots,n$ and $k=0,\cdots, 2^{h_n}-1$, let
\small
\[B_i(x_i;h_n,k)=\min_{j=0,\cdots,2^{h_i}-1}\hspace{-.1cm}\big\{\T_{h_i,j}(x_i):\T_{h_i,j}(x_i)>a_{h_n,k} \mbox{ and }a_{h_i,j}<a_{h_n,k+1} \big\}.\]
\normalsize
For each $k$ and $i$, $B_i(x_i;h_n,k)$ is either the smallest $\T_{h_i,j}(x_i)$ which is greater than $a_{h_n,k}$ and such that the correspondent $A_{h_i,j}(x_i)$ has non-empty intersection with $A_{h_n,k}(x_n)$, or empty. Let
\small
\[c_{i}(x_i;h_n,k)=\left\{
\begin{array}{cc}
  a_{h_n,k}, & \mbox{ if } \ \  B_i(x_i;h_n,k)=\emptyset; \\
  B_i(x_i;h_n,k), & \mbox{otherwise}.
\end{array}\right.\]
\normalsize
Then, for each $k=1,\cdots,2^{h_n}-1$, setting $b_{h_n,k}(x_{2:n})=\displaystyle{\min_{i=2,\cdots,n}}\big\{c_{i}(x_i;h_n,k)\big\}$, it follows that
\small
\[\widetilde{A}_{h_n,k}(x_{2:n})=\big[a_{h_n,k}, b_{h_n,k}(x_{2:n})\big],\]
\normalsize
which is a closed subset of $[a_{h_n,k},a_{h_n,k+1}]$. Also notice that, from the definition of $b_{h_n,k}(x_{2:n})$, we could have $\widetilde{A}_{h_n,k}(x_{2:n})=\{a_{h_n,k}\}$, in which case we set $\widetilde{A}_{h_n,k}(x_{2:n})=\emptyset$ (although from a measure-theoretical point of view, this correction makes no difference). Again we are omitting the dependence in $s$ from the notation on both, $b_{h_n,k}$ and $\widetilde{A}_{h_n,k}$. Each $b_{h_n,k}(x_{2:n})$ above determines the smallest $\mathcal{T}_{h_i,j}(x_i)$ that lies on the $k$-th node of $T^{h_n}_s$ (which has the smallest nodes among all $T^{h_i}_s$'s), so that $\widetilde{A}_{h_n,k}$'s are just the intersection of all $A_{h_i,k}(x_i)$'s with end point in the $k$-th node of $T^{h_n}_s$. Also notice that the $\widetilde{A}_{h_n,k}$'s are pairwise disjoints. One can rewrite \eqref{aux2} as
\small
\begin{equation}\label{aux3}
H_{t_1,\cdots,t_n}(x_1,\cdots,x_n)=\P\Big(U_0\in\cup_{k=0}^{2^{h_n}-1}\big[\widetilde{A}_{h_n,k}(x_{2:n})\cap[0,x_1]\big]\Big).
\end{equation}
\normalsize
Now, let $n_1:=n_1(x_1;n)=\big\{k:x_1\in[a_{h_n,k},a_{h_n,k+1})\big\}\in\{0,\cdots,2^{h_n}-1\}$, and assume for the moment that $n_1\geq1$. Then \eqref{aux3} becomes
\small
\begin{align*}
H_{t_1,\cdots,t_n}(x_1,\cdots,x_n)&= \sum_{k=0}^{n_1-1}  \P\big(U_0 \in \widetilde{A}_{h_n,k}(x_{2:n})\big) +
\P\big(U_0 \in \widetilde{A}_{h_n,n_1} (x_{2:n})\cap [a_{h_n,n_1},x_1]\big)\\
&=   \sum_{k=0}^{n_1-1}\mu_s\big([a_{h_n,k},b_{h_n,k}(x_{2:n})]\big)+\mu_s\big([a_{h_n,n_1},\min\{x_1,b_{h_n,n_1}   (x_{2:n})\}]\big)\\
&=\!    \sum_{k=0}^{n_1-1}   \big[F_0\big(b_{h_n,k}(x_{2:n})\big) - F_0(a_{h_n,k})\big]  + \min\big\{F_0(x_1),F_0(b_{h_n,n_1} (x_{2:n}))\big\} - F_0(a_{h_n,n_1}).
\end{align*}
\normalsize
If $n_1=0$, then
\small
\[H_{t_1,\cdots,t_n}(x_1,\cdots,x_n)=\min\big\{F_0(x_1),F_0(b_{h_n,0}(x_{2:n}))\big\}-F_0(a_{h_n,0}).\]
\normalsize
In any case, we can write
\small
\begin{eqnarray*}
H_{t_1,\cdots,t_n}(x_1,\cdots,x_n)&=&\left(\sum_{k=0}^{n_1-1} F_0\big(b_{h_n,k}(x_{2:n})\big)-F_0(a_{h_n,k})\!\right)\!\delta_{\N^\ast}\!(n_1) +\\ &&\hskip.7cm+\,\min\big\{F_0(x_1),F_0(b_{h_n,n_1}(x_{2:n}))\big\}-F_0(a_{h_n,n_1}).
\end{eqnarray*}
\normalsize
 Recall that the distribution function of $T^t_s(U_0)$ is also $F_0$ by the $T_s$-invariance of $\mu_s$. Now applying Sklar's Theorem, it follows that,
\small
\begin{align*}
\hskip-2cmC_{X_{t_1},\cdots,X_{t_n}}(u_1,\cdots,u_n)&=H_{t_1,\cdots,t_n}\big(F_0^{-1}(u_1),\cdots, F_0^{-1}(u_n)\big)\\
&=\left(\sum_{k=0}^{n_0-1} F_0\Big(b_{h_n,k}\big(F^{-1}_0(u_{2:n})\big)\Big)-F_0(a_{h_n,k})\!\right)\!\delta_{\N^\ast}\!(n_1)  +\\
&\hskip1cm+\min\big\{u_1,F_0\big(b_{h_n,n_0}\big(F_0^{-1}(u_{2:n})\big)\big)\big\}-F_0(a_{h_n,n_0}).
\end{align*}
\normalsize
where  $n_0:=n_1\big(F_0^{-1}(u_1),n\big)=\big\{k:u_1\in\big[F_0(a_{h_n,k}),F_0(a_{h_n,k+1})\big)\big\}$,
which is the desired formula.\fim
\end{proof}
\begin{rmk}
Notice that the proof of Theorem \ref{mdim} from equation \eqref{aux3} on is exactly the same as the one in Proposition \ref{copcr} with the obvious notational adaptations.
\end{rmk}
Now we turn our attention to the case where $\p$ is an almost everywhere decreasing function. In view of Theorem \ref{ver}, one cannot expect a simple expression for the copula. What happens is that the copula in this case will be the sum of the lower dimensions copulas related to the iterations $T_s^k(U_0)$, as the next proposition shows.

        \begin{prop}\label{mdec}
        Let $\{X_n\}_{n\in\N}$ be an MP process with parameter $s\in(0,1)$, and $\phi\in\L^1(\mu_s)$ be an almost everywhere decreasing function. Let  $t,h_1,\cdots,h_n\in\N$, $0<h_1<\cdots<h_n$ and set $Y_0:=U_0$ and $Y_k:=T_s^{h_k}(U_0)$. Denote the copula associated to the random vector $(X_t,X_{t+h_1},\cdots,X_{t+h_n})$ by $C_{\boldsymbol t}$. Then the following relation holds
        \small
        \begin{eqnarray}\label{mrel}
        &&\hskip-.6cm C_{\boldsymbol t}(u_0,\cdots,u_n)=1-n+\sum_{i=0}^nu_i +\sum_{i=0}^n\sum_{j=i+1}^n C_{Y_i,Y_j}(1-u_i,1-u_j)+\cdots+\nonumber\\
        &&+\,(-1)^{n-1}\sum_{k_1=0}^n\sum_{k_2=k_1+1}^n\cdots\!\!\!\sum_{k_{n-1}=k_{n-2}+1}^n\!\!\!\!\!
        C_{Y_{k_1},\cdots,Y_{k_{n-1}}}(1-u_{k_1},\cdots,1-u_{k_{n-1}})+\nonumber\\
        &&\hspace{1cm}+\,(-1)^{n}C_{U_0,Y_1,\cdots,Y_n}(1-u_0,\cdots,1-u_n),
        \end{eqnarray}
        \normalsize
        everywhere in $I^{n+1}$.
        \end{prop}

\begin{proof}
Let  $t,h_1,\cdots,h_n\in\N$, $0<h_1<\cdots<h_n$, $t\neq 0$. Set $Y_0:=U_0$,  $Y_k:=T_s^{h_k}(U_0)$ and  $y_k:=\p(x_k)$. We have
\small
\begin{eqnarray}\label{chain}
&&\hspace{-1.5cm}H_{X_0, X_{h_1},\cdots,X_{h_n}}(x_0,x_1,\cdots,x_n)=\P\big(U_0\geq y_0, Y_1\geq y_1,\cdots,Y_n\geq y_n \big)\nonumber\\
&=&\P\Big(U_0\geq y_0\big|Y_1\geq y_1 ,Y_2\geq y_2,\cdots,Y_n\geq y_n\Big) \P\Big(Y_1\geq y_1 ,\cdots,Y_n\geq y_n\Big)\nonumber\\
&=& \P\Big(Y_1\geq y_1 ,\cdots,Y_n\geq y_n\Big)-\P\big(U_0\leq y_0, Y_1\geq y_1,\cdots,Y_n\geq y_n \big).
\end{eqnarray}
\normalsize
Upon applying a long chain of a conditioning argument on both terms in \eqref{chain}, we arrive at
\small
\begin{eqnarray}\label{chain2}
&&\hspace{-.6cm}H_{X_0, X_{h_1},\cdots,X_{h_n}}(x_0,x_1,\cdots,x_n)=1-\sum_{i=0}^nF_0(y_i)+\sum_{i=0}^n\sum_{j=i+1}^nH_{Y_i,Y_j}(y_i,y_j)+\nonumber\\
&& \hspace{0.7cm}+\cdots+\,(-1)^{n-1}\sum_{k_1=0}^n\sum_{k_2=k_1+1}^n\cdots\!\!\!\sum_{k_{n-1}=k_{n-2}+1}^n\!\!\!\!\!
H_{Y_{k_1},\cdots,Y_{k_{n-1}}}(y_{k_1},\cdots,y_{k_{n-1}})+\nonumber\\
&&\hspace{0.7cm}+\,(-1)^{n}H_{U_0,Y_1,\cdots,Y_n}(y_0,\cdots,y_n).
\end{eqnarray}
\normalsize
A simple calculation (using the $T_s$-invariance of $\mu_s$) shows that, for all $t\!\in\!\N^\ast$ and $x\in (0,1)$,
\small
\[F_{X_t}(x)=1-F_0\big(\p(x)\big) \ \ \ \mbox{ and }\ \ \ \ F_{X_t}^{-1}(x)=\ip\big(F_0^{-1}(1-x)\big), \]
\normalsize
so that, the result follows upon applying Sklar's Theorem to \eqref{chain2} (recall that $y_k=\p(x_k)$).\fim
\end{proof}

\begin{rmk}
Notice that the copula in Proposition \ref{mdec} can be explicitly calculated since \eqref{mrel} is written as sums of the copulas of vectors containing $U_0$ and $T^t(U_0)$ for different $t$'s, so that the desired formulas can be deduced in terms of the copulas in Theorem \ref{mdim}.
\end{rmk}

\section{Numerical Approximations to the MP copulas}
\setcounter{equation}{0}
The MP copulas derived in the last sections do not have readily computable formulas, especially because $\mu_s$ does not have explicit expression and because even apparently simple tasks like determining the discontinuity points of $T^h_s$ or to compute explicit formulas for the branches of $T^h_s$ can be highly complex ones. However, one can still study the copulas derived in the last sections by using appropriate approximations to the functions appearing in the copula expression. Besides the invariant measure $\mu_s$, computation of the bidimensional copulas so far discussed also involves computation of the quantile function $F_0^{-1}$, the inverse of $T^h_s$ and the end points $\{a_{h,k}\}_{k=0}^{2^h}$ of the nodes of $T_s^h$.

In this section our goal is to derive simple approximations to these functions in order to obtain an approximation to the copula itself, which we shall prove to converge uniformly in its arguments to the true copula. The approximations presented here are simple ones, usually a linear interpolation based on a grid of values, but the technique and results we shall use and prove here are stronger and cover a wide range of approximations, for instance, all results hold if we use some type of spline interpolation instead of a linear one. This is so because the functions to be approximated are generally very smooth. We also evaluate the stability and performance of the approximations by simple numerical experiments.
\subsection*{ Approximation to $\mu_s$}
We start with an approximation to $\mu_s$. In this direction there are at least two ways to compute approximations to $\mu_s$. One way is by using the ideas and results outlined in Dellnitz and Junge (1999), which are based on a discretization of the Perron-Frobenius operator by means of a Garlekin projection type approximation in order to compute the eigenvectors of the discretized operator corresponding to the eingenvalue 1. Although it can be used to approximate any SBR measure, the method is especially suited to approximate and study (almost) cyclical behavior of dynamical systems. However, its complexity makes the efficient implementation troublesome. A much simpler idea, which we shall adopt here, is to approximate the measure by truncating equation \eqref{sbr} for a reasonably large value of $n$. That is, we consider the approximating measure
\vskip-1\baselineskip

\small
\begin{equation}\label{app}
\mu_n(A;s,x_0)=\frac{1}{n}\sum_{k=0}^{n-1}\delta_{T^k_s(x_0)}(A)
\end{equation}
\normalsize
which converges in a weak sense to $\mu_s$ as $n$ tends to infinity, for almost all initial points $x_0\in I$ and all $\mu_s$-continuity sets $A$. The iterations of $T_s$ are known to be unstable with respect to the initial point in the sense that, given a small $\eps>0$ and a point $x\in(0,1)$, the trajectories $T^k_s(x)$ and $T_s^k(x+\eps)$ become far apart exponentially fast. The approximation \eqref{app}, however, is quite stable with respect to the initial point $x_0$ for large $n$. For instance, in Figure \ref{f1} we show the measure of the sets $[0.1,0.2]$ and $[0.4,0.6]$ obtained by using $\mu_n(\cdot;s,x_0)$ with $s=0.5$, for 50 different initial points $x_0$ and 3 different truncation points $n\in\{300,000; 1,000,000; 3,000,000\}$.  All plots are in the same scale (within set) in order to make comparison possible. In Table \ref{t1} we show basic statistics related to Figure \ref{f1}. Notice that, in average, the 1,000,000 and 3,000,000 iteration cases are very similar and all cases are fairly stable with respect to the initial points (observe the scale).
\begin{figure}[!h]
\centering
\includegraphics[scale=0.7]{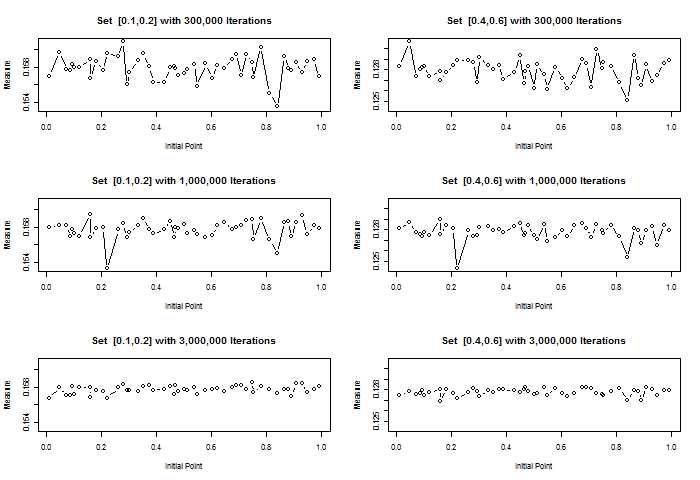}
  \caption{{\small Performance of the approximation \eqref{app} for truncation points $n\in\{300,000; 1,000,000; 3,000,000\}$ (top, middle and bottom, respectively) and 50 different initial points for $s=0.5$. The measured sets are (left) $[0.1,0.2]$ and (right) $[0.4,0.6]$. All plots within the same set are in the same scale.}}\label{f1}
\end{figure}
\begin{table}[htb]
\caption{{\small Summary statistics for the data presented in Figure \ref{f1}.}}\vspace{.3cm}\label{t1}
\centering
{\scriptsize
\begin{tabular}{|c|c|c|c|c|}
\hline
\hline
Set & $n$ &     300,000 &    1,000,000 &    3,000,000 \\
\hline
\hline
\multirow{3}{*}{\begin{sideways} [0.2,0.3] \end{sideways} }&$[\min,\max]$ & [0.12511,0.13067] & [0.12431,0.12901] & [0.12688,0.12825] \\
   &  range &    0.00556 &    0.00470 &    0.00137 \\
      &mean &    0.12790 &    0.12775 &    0.12777 \\
\hline
\hline
\multirow{3}{*}{\begin{sideways} [0.4,0.6] \end{sideways} }&$[\min,\max]$ &  [0.15349,0.16092]  & [0.15326,0.15944] & [0.15676,0.15857] \\
&     range &    0.00743 &    0.00618 &    0.00181 \\
&   mean &    0.15792 &    0.15771 &    0.15771 \\
\hline
\hline
\end{tabular}}
\end{table}

Next question is how good is the approximation \eqref{app}? One way to test this is by testing whether the approximation is invariant under $T_s$. For given initial points, say $x_1,\cdots,x_k$ and some interval $[a,b]$, we calculate $\mu_n\big([a,b];s,x_i\big)$ and $\mu_n\big(T_s^{-1}([a,b]);s,x_j\big)$. If the difference between the two quantities is small for different pairs $(x_i,x_j)$, one can conclude that the approximation is reasonably good. In Table \ref{t2} we present the difference $\big|\mu_n\big([a,b];s,x_i\big)-\mu_n\big(T_s^{-1}([a,b]);s,x_j\big)\big|$ for 7 different initial points and 3 different sets $[a,b]$. The truncation point was taken to be 3,000,000 and  $s=0.5$. From Table \ref{t2} we conclude that the approximation \eqref{app} performs very well in all cases and that it can be taken to be $T_s$-invariant. As expected, when $x_i=x_j$ the differences are the smallest  ($<10^{-8}$ in all cases).
\begin{table}[!h]
\caption{{\small Difference $\big|\mu_n\big([a,b];s,x_i\big)-\mu_n\big(T_s^{-1}([a,b]);s,x_j\big)\big|$ for different values of $x_0$ and sets $[a,b]$. The truncation point was taken to be $n=3,000,000$ and $s=0.5$. The initial points are $(x_1,\cdots,x_7)=\big(\pi,\pi/(\sqrt2+1),\pi\sqrt2,\pi+\sqrt2,\sqrt7,\pi+\sqrt7,\sqrt{11}+\sqrt7\big)(\mathrm{mod}\ 1)$.}}\vspace{.3cm}\label{t2}
\centering
{\tiny
\begin{tabular}{|c|c|c|c|c|c|c|c|c|}
\hline
\hline
           &    initial &     $x_1$ &     $x_2$ &     $x_3$ &     $x_4$ &     $x_5$ &     $x_6$ &     $x_7$ \\
           \hline
           \hline
\multirow{7}{*}{\begin{sideways} [0.05,0.2] \end{sideways} }
           &     $x_1$ &          0.00000 &    0.00019 &    0.00040&    0.00008 &    0.00004 &   0.00062 &    0.00022 \\

           &     $x_2$ &    0.00019 &          0.00000 &    0.00020 &    0.00027 &    0.00024 &   0.00043 &    0.00042 \\

           &     $x_3$ &     0.00040 &     0.00000 &          0.00000 &    0.00047 &    0.00044 &   0.00022 &    0.00062 \\

           &     $x_4$ &   0.00008 &    0.00030 &   0.00047 &          0.00000 &   0.00003 &    0.00070 &    0.00015 \\

           &     $x_5$ &   0.00004 &    0.00020 &   0.00044 &    0.00003 &          0.00000 &   0.00066 &    0.00018 \\

           &     $x_6$ &    0.00062 &    0.00043 &    0.00022 &     0.00070 &    0.00066 &          0.00000 &    0.00084 \\

           &     $x_7$ &   0.00022 &    0.0004 &   0.00062 &   0.00015 &   0.00018 &   0.00084 &          0.00000 \\
           \hline
           \hline

           &    initial &     $x_1$ &     $x_2$ &     $x_3$ &     $x_4$ &     $x_5$ &     $x_6$ &     $x_7$ \\
           \hline
           \hline
\multirow{7}{*}{\begin{sideways} [0.3,0.8] \end{sideways} }
           &     $x_1$ &    0.00000 &    0.00019 &    0.00011 &   0.00009 &    0.00052 &    0.00036 &    0.00155 \\

           &     $x_2$ &   0.00019 &    0.00000 &   0.00008 &   0.00028 &    0.00033 &    0.00016 &    0.00136 \\

           &     $x_3$ &   0.00011 &    0.00008 &    0.00000 &   0.00020 &    0.00041 &    0.00024 &    0.00144 \\

           &     $x_4$ &    0.00009 &    0.00028 &    0.00020 &    0.00000 &    0.00061 &    0.00045 &    0.00164 \\

           &     $x_5$ &   0.00052 &   0.00033 &   0.00041 &   0.00061 &    0.00000 &   0.00016 &    0.00103 \\

           &     $x_6$ &   0.00036 &   0.00016 &   0.00024 &   0.00045 &    0.00016 &    0.00000 &    0.00119 \\

           &     $x_7$ &   0.00155 &   0.00136 &   0.00144 &   0.00164 &   0.00103 &   0.00119 &    0.00000 \\
           \hline\hline

           &    initial &     $x_1$ &     $x_2$ &     $x_3$ &     $x_4$ &     $x_5$ &     $x_6$ &     $x_7$ \\
           \hline
           \hline
\multirow{7}{*}{\begin{sideways} [0.7,0.95] \end{sideways} }

           &     $x_1$           &    0.00000 &    0.00011 &    0.00005 &    0.00012 &    0.00003 &    0.00012 &    0.00089 \\

           &     $x_2$            &    0.00011 &    0.00000 &    0.00016 &    0.00022 &    0.00013 &    0.00022 &    0.00078 \\

           &     $x_3$            &    0.00005 &    0.00016 &    0.00000 &    0.00006 &    0.00003 &    0.00006 &    0.00094 \\

           &     $x_4$            &    0.00012 &    0.00022 &    0.00006 &    0.00000 &    0.00009 &    0.00000 &    0.00100 \\

           &     $x_5$            &    0.00003 &    0.00013 &    0.00003 &    0.00009 &    0.00000 &    0.00009 &    0.00091 \\

           &     $x_6$            &    0.00012 &    0.00022 &    0.00006 &    0.00000 &    0.00009 &    0.00000 &    0.00101 \\

           &     $x_7$            &    0.00089 &    0.00078 &    0.00094 &    0.00100 &    0.00091 &    0.00101 &    0.00000 \\
\hline
\hline
\end{tabular}}
\end{table}

In the remaining of this section we shall assume that $s\in(0,1)$ has been fixed and $x_0\in(0,1)$ has been chosen so that the approximation \eqref{app} converges to $\mu_s$. Since no confusion will arise, we shall drop $s$ and $x_0$ from the notation and write the approximation \eqref{app}, based on a size $n$ iteration vector, just by $\mu_n(\cdot)$.
\subsection*{Approximating $F_0^{-1}$ and  the nodes of $T^h_s$}
\setcounter{footnote}{1}
In order to approximate $F^{-1}_0$, one can use an empirical version based on the same iteration vector from which $\mu_n$ is derived. First we need to define an approximation to $F_0$ from which an approximation to $F_0^{-1}$ will be derived. Let $\widehat{F}_n$ be the empirical distribution based on a size $n$ iteration vector  $\big(x_0,T_s(x_0),\cdots,T^{n-1}_s(x_0)\big)$ and let $x_1,\cdots,x_n$ be the jump points\footnote{by the choice of $x_0$, there will be exactly $n$ jump points.} of $\widehat{F}_n$. Consider the set
$L_n:=\{0=x_0,x_1,\cdots,x_n,x_{n+1}=1\}$. Given $x\in I\backslash L_n$, there exists a $k\in\{0,\cdots,n\}$ such that $x\in(x_k,x_{k+1})$. We define the approximate value of $F_0(x)$, denoted by $F_n(x)$, as the linear interpolation of $x$ between the points $\big(x_{k},\widehat{F}_n(x_{k})\big)$ and $\big(x_{k+1},\widehat{F}_n(x_{k+1})\big)$, that is, we set
\small
\begin{equation}\label{fn}
F_n(x):= \left(\frac{\widehat{F}_n(x_{k+1})-\widehat{F}_n(x_{k})}{x_{k+1}-x_k}\right)x+\frac{\widehat{F}_n(x_{k})x_{k+1}-
\widehat{F}_n(x_{k+1})x_k}{x_{k+1}-x_k}.
\end{equation}
\normalsize
 If $x\in L_n$, we simply define $F_n(x):=\widehat{F}_n(x)$. Notice that, for each $n$, $F_n:I\rightarrow I$ is a one-to-one, increasing and uniformly continuous function, so that its inverse, $F_n^{-1}$, is well defined and is also one-to-one and uniformly continuous. In the next  proposition, we show that $F_n(x)\rightarrow F_0(x)$ and $F_n^{-1}(x)\rightarrow F_0^{-1}(x)$, both limits being uniform in $x$.

        \begin{prop}
        Let $\widehat{F}_n$ be the empirical distribution based on an iteration vector $\big( x_0 , T_s(x_0) ,$ $ \cdots , T^{n-1}_s(x_0) \big)$ and let $x_1,\cdots,x_n$ be the jump points of $\widehat{F}_n$. Let $F_n$ be the approximation \eqref{fn} based on $\{x_1,\cdots,x_n\}$ and $F_n^{-1}$ be its inverse. Then,
        \small
        \[F_n(x)\longrightarrow F_0(x) \ \ \ \mbox{ and } \ \ \  F_n^{-1}(x)\longrightarrow F_0^{-1}(x),\]
        \normalsize
        uniformly in $x$.
        \end{prop}
\begin{proof}
By the Glivenko-Cantelli theorem, $\widehat{F}_n(x)\rightarrow F_0(x)$ uniformly in $x\in[0,1]$, so that, given $\eps>0$, one can find $n_0:=n_0(\eps)>0$ such that
if $n>n_0$, then $\big|\widehat{F}_n(x)-F_0(x)\big|<\eps$ uniformly in $x$. Now, for $x\in(0,1)$ (if $x$ equals 0 or 1, the result is trivial), there exists a $k\in\{1,\cdots,n\}$ such that $x\in[x_k,x_{k+1})$. Hence, if $n>n_0$
\small
\begin{eqnarray*}
\big|F_n(x)-F_0(x)\big|&\leq&\big|F_n(x)-\widehat{F}_n(x)\big|+\big|\widehat{F}_n(x)-F_0(x)\big|\\
&<&\big|\widehat{F}_n(x_{k+1})-\widehat{F}_n(x_k)\big|+\eps\\
&\leq&\sup_{i=1,\cdots,n-1}\big\{\big|\widehat{F}_n(x_{i+1})-\widehat{F}_n(x_i)\big|\big\}+\eps\\
&\leq&\frac{1}{n}+\eps,
\end{eqnarray*}
\normalsize
uniformly in $x$. To show the convergence of the inverse, let $y\in[0,1]$ and  $\eps>0$ be given and notice that $F_n^{-1}$ being uniformly continuous, one can find a $\delta:=\delta(\eps)>0$ such that
\small
 \[|x-y|<\delta \ \Longrightarrow |F_n^{-1}(x)-F_n^{-1}(y)|< \eps.\]
 \normalsize
 Now, since $F_n$ converges uniformly to $F_0$, there exists $n_1:=n_1(\eps)>0$ such that,
 \small
 \[n>n_1\ \Longrightarrow \big|F_n(x)-F_0(x)\big|<\delta,\]
 \normalsize
for all $x\in I$. Also, since $F_0$ is one to one, there exists $v_0\in[0,1]$ such that $y=F_0(v_0)$. Therefore, if $n>n_1$
\small
\[\big|F_n^{-1}(y)-F_0^{-1}(y)\big|=\big|F_n^{-1}\big(F_0(v_0 )\big)-v_0 \big| =\big|F_n^{-1}\big(F_0(v_0 )\big)-F_n^{-1}\big(F_n(v_0 )\big)\big|<\eps \]\normalsize
and since $n_1$ is independent of $y$, the desired convergence follows.\fim
\end{proof}
As for the end points $\{a_{h,k}\}_{k=0}^{2^h}$ of the nodes of $T^h_s$, let $\{x_1,\cdots,x_m\}\in (0,1)$, $x_i\neq x_j$ and consider the set $\{T^h_s(x_1),\cdots,T^h_s(x_m)\}$, for $m>0$ sufficiently large\footnote{By ``sufficiently large'' we mean that $m$ should be at least large enough to guarantee that the set $\{T^h_s(x_1),\cdots,T^h_s(x_m)\}$ reflects the $2^h-1$ discontinuities of $T^h_s$, or, in other words, $m\geq 2^h$. The limits in $m$ taken for an approximation are understood to be in terms of partitions, that is, we start with a sufficiently large set of points, say $I_m=\{x_1,\cdots,x_m\}$ and consider refinements of the form $I_{m+1}=I_m\cup\{x_{m+1}\},\cdots, I_{m+k}=I_{m+k-1}\cup\{x_{m+k}\}$. Suppose that $R_m:=R(I_m)$ is an approximation based on $I_m$. For a sequence of refinements $\{I_k\}_{k=m+1}^\infty$ we consider the sequence $\{R(I_k)\}_{k=m+1}^\infty$. Whenever the last limit exists, we set $\displaystyle{\lim_{m\rightarrow\infty} R_m= \lim_{k\rightarrow\infty} R(I_k)}$.}. Note that $a_{h,0}=0$ and $a_{h,2^h}=1$, for any $h$. Let $D=\big\{i:T^h_s(x_i)>T^h_s(x_{i+1})\big\}\subset \{1,\cdots,m\}$. The set $D$ contains the indexes $i\in\{1,\cdots,m\}$ for which the interval $[x_i,x_{i+1}]$ contains a discontinuity of $T^h_s$. Let $\{d_j\}_{j=1}^{2^h-1}$ denote the ordered elements of $D$, so that the interval $[x_{d_j},x_{d_{j+1}}]$ contains the $j$-th discontinuity  of $T^h_s$. Now consider the function $T_{i,h;s}^\ast:[x_{d_i},x_{d_i+1}]\rightarrow [0,2]$ given by $T_{i,h}^\ast(x;s):= T_s^{h-1}(x)+\big(T_s^{h-1}(x)\big)^{1+s}$ and notice that we can write
\small
\[T_s^{h}(x)= T_{i,h}^\ast(x;s) -\delta_{[1,2]}\big(T_{i,h}^\ast(x;s)\big).\]
\normalsize
Since there is a discontinuity of $T_s^h$ in the interval $[x_{d_i},x_{d_i+1}]$, we have $T_{i,h}^\ast(x_{d_i};s)\leq1$ and $T_{i,h}^\ast(x_{d_i+1};s)\geq1$ and since $T_{i,h}^\ast$ is continuous and increasing, there exists a point $x\in[x_{d_i},x_{d_i+1}]$ such that $T_s^\ast(x;s)=1$, which is precisely $a_{h,i}$. With this in mind, let $a_{h,i}^m$ denote the approximation to $a_{h,i}$ obtained from $\{x_1,\cdots,x_m\}$ by using a linear interpolation between the points $\big(x_{d_i},T_{i,h}^\ast(x_{d_i};s)\big)$ and $\big(x_{d_i+1},T_{i,h}^\ast(x_{d_i+1};s)\big)$. That is, $a_{h,i}^m$ is given by
\small
\begin{equation}\label{ahkapp}
a_{h,i}^m=x_{d_i} +\frac{x_{d_i+1}-x_{d_i}}{T_{i,h}^\ast(x_{d_i+1};s)-T_{i,h}^\ast(x_{d_i};s)}\,\big(1-T_{i,h}^\ast(x_{d_i};s)\big),
\end{equation}
\normalsize
for all $d_i\in D$. Clearly $a_{h,i}^m\underset{m\rightarrow\infty}{\longrightarrow} a_{h,i}$, since $|x_{d_i+1}-x_{d_i}|\underset{m\rightarrow\infty}{\longrightarrow} 0$ and by the continuity of $T_{i,h}^\ast$, for each $i\in\{1,\cdots,2^h-1\}$.
\subsection*{Approximating $\T_{h,k}$}
Concerning the approximation of $\T_{h,k}$, we shall use an argument based on an empirical inverse and linear interpolation, but we shall also need a doubling argument in order to improve accuracy of the approximation near the discontinuities and guarantee the uniform convergence of the approximation to its target. So let $\{0=x_1,\cdots,x_m=1\}\in I$, $x_i<x_j$ and consider the set $\big\{T^h_s(x_1),\cdots,T^h_s(x_m)\big\}$, for $m>0$ sufficiently large. Given $y\in[0,1]$, recall that the inverse of $T^h_s(y)$ is a size $2^h$ vector which we denoted by $\big(\T_{h,0},\cdots,\T_{h,2^h-1}\big)$.
Let again $D=\big\{i:T^h_s(x_i)>T^h_s(x_{i+1})\big\}\subset \{1,\cdots,m\}$ and $\{d_i\}_{i=1}^{2^h-1}$ be the ordered points in $D$. Suppose that we know exactly or have good estimates for the nodes $\{a_{h,k}\}_{k=0}^{2^h}$ of $T_s^h$ (for instance, we could use $\{a_{h,k}^m\}_{k=0}^{2^h}$, as described before, based on the same set $\{x_1,\cdots,x_m\}$ considered here). For $i=0,\cdots,2^h-1$, let
\small
\[R_{h,i}=\{x_{h,i}^{(1)}, \cdots, x_{h,i}^{(p_i)}\}:=\{a_{h,i}^m,x_{d_i+1},\cdots,x_{d_{i+1}}, a_{h,i+1}^m \}\]
\normalsize
and
\small
\[I_{h,i}=\{y_{h,i}^{(1)}, \cdots, y_{h,i}^{(p_i)}\}:=\big\{0,T^h_s(x_{d_i+1}),\cdots,T^h_s(x_{d_{i+1}}), 1 \big\}.\]
\normalsize
Given $y\in[0,1]$, for each $i=0,\cdots,2^h-1$, there exists a {\small$y_{h,i}^{(k)}\in I_{h,i}$} such that \small$y\in\big[y_{h,i}^{(k)},y_{h,i}^{(k+1)}\big)$\normalsize. We define the approximation $\T_{h,i}^m(y)$ of $\T_{h,i}(y)$, as being the linear interpolation of $y$ between the points $\big(x_{h,i}^{(k)},y_{h,i}^{(k)}\big)$ and \small$\big(x_{h,i}^{(k+1)}\!\!\!,\,\,\,y_{h,i}^{(k+1)}\big)$\normalsize. That is, for each  $i=0,\cdots,2^h-1$,
\small
\begin{equation}\label{tauapp}
\T_{h,i}^m(y)=x_{h,i}^{(k)}+\frac{x_{h,i}^{(k+1)}-x_{h,i}^{(k)}}{y_{h,i}^{(k+1)}-y_{h,i}^{(k)}}\,\big(y-y_{h,i}^{(k)}\big).
\end{equation}
\normalsize
Notice that if $y$ equals 0 or 1, we have $\T_{h,i}^m(y)=\T_{h,i}(y)$. Also, as the partition $\{x_1,\cdots,x_m\}$ increases, $\big|x_{k+1}-x_k\big|\underset{m\rightarrow\infty}{\longrightarrow} 0$  and the uniform continuity of $T_s^h$ clearly implies $\T_{h,i}^m(y)\underset{m\rightarrow\infty}{\longrightarrow} \T_{h,i}(y)$, for each $y\in[0,1]$, for $i=0,\cdots,2^h-1$. More is true: the convergence is actually uniform in $y$, as we show in the next proposition.

        \begin{prop}
        Let $\T_{h,k}^m$ be the approximation of $\T_{h,k}$ given by \eqref{tauapp} based on a partition $R_m$. Then,
        \small
        \[\T_{h,k}^m(y)\longrightarrow\T_{h,k}(y),\]
        \normalsize
        for each $k=0,\cdots,2^h-1$, as $m$ goes to infinity (that is, as the partition gets thinner). Moreover, the convergence is uniform in $y\in[0,1]$.
        \end{prop}

\begin{proof}
 Given $\eps>0$, the uniform continuity of $\T_{h,k}$ implies the existence of a $\delta:=\delta(\eps)>0$ such that
\small
\[|x-y|<\delta\Longrightarrow \big|\T_{h,k}(x)-\T_{h,k}(y)\big|<\eps,\]
\normalsize
for all $x\in[0,1]$. Let $R_0=\{0=x_1,\cdots,x_{m_0}=1\}\in I$ for a sufficiently large $m_0\in\N^\ast$ such that
\small
\[\sup_{i=1,\cdots,m_0-1}\big\{|x_{i+1}-x_i\big|\}<\delta.\]
\normalsize
For $m>m_0$, let $R_m=\{x_1^\ast,\cdots,x_m^\ast\}\supset R_0$ be a size $m$ refinement of $R_0$. Given $y\in(0,1)$, for each $i=0,\cdots,2^h-1$, let $\T_{h,i}^m$ be the approximation \eqref{tauapp} based on $R_m$. By construction and since $y\in(0,1)$, it follows that
\small
\[\T_{h,i}\big(x_{h,i}^{(k)}\big)\leq\T_{h,i}^m(y)<\T_{h,i}\big(x_{h,i}^{(k+1)}\big) \ \ \mbox{ and } \ \ \T_{h,i}\big(x_{h,i}^{(k)}\big)\leq\T_{h,i}(y)<\T_{h,i}\big(x_{h,i}^{(k+1)}\big),\]
\normalsize
so that
\small
\begin{eqnarray*}
\big|\T_{h,i}^m(y)-\T_{h,i}(y)\big|&\leq& |\T_{h,i}\big(x_{h,i}^{(k+1)}\big)-\T_{h,i}\big(x_{h,i}^{(k)}\big)|\\
&\leq&\sup_{j=1,\cdots,m-1}\big\{\big|\T_{h,i}(x_{j+1})-\T_{h,i}(x_j)\big|\big\}< \eps,
\end{eqnarray*}
\normalsize
for all $y\in(0,1)$. If $y\in\{0,1\}$, by construction $\T_{h,i}(y)=\T_{h,i}^m(y)$, so that the result follows uniformly for all $y\in[0,1]$, as desired.\fim
\end{proof}

\subsection{Approximating the lag $h$ MP copula}
With these approximations in hand, we can now define the approximation for the copula $C_{X_t,X_{t+h}}$ when $\p$ is almost everywhere increasing given in Proposition \ref{copcr} but in the form \eqref{copmu}. For $(u,v)\in I^2$, $n>0$ and $m\geq2^h$, we set
\small
\begin{align}\label{copapp}
C_{m,n}(u,v;h)&=\bigg(\sum_{k=0}^{n_0^\ast-1}\mu_n\Big(\big[a_{h,k}^m,\T_{h,k}^m\big( F^{-1}_n(v)\big)\big]\Big)\bigg)\delta_{\N^\ast}\!(n_0^\ast) +\nonumber\\
&\hspace{2.1cm}+\mu_n\Big(\big[a_{h,n_0}^m,\min\big\{F_n^{-1}(u),\T_{h,n_0}^m\big( F^{-1}_n(v)\big)\big\}\big]\Big),
\end{align}
\normalsize
where $n_0^\ast:=n_0(m,n)=\big\{k:u\in\big[F_n(a_{h,k}^m),F_n(a_{h,k+1}^m)\big)\big\}$ and $\displaystyle{\lim_{m,n\rightarrow\infty}} n_0^\ast=n_0$ since $F_n$ converges uniformly to $F_0$ and $a_{h,k}^m$ converges to $a_{h,k}$. In the next theorem we establish the convergence of the approximation \eqref{copapp} to the true copula.

        \begin{thm}\label{conv}
        Let $C_{m,n}(u,v;h)$ be given by \eqref{copapp}. Then, for all $(u,v)\in I^2$, $t\geq0$ and $h>0$
        \small
        \[\lim_{n\rightarrow\infty}\lim_{m\rightarrow\infty}C_{m,n}(u,v;h)=\lim_{m\rightarrow\infty}\lim_{n\rightarrow\infty}C_{m,n}(u,v;h)= \lim_{m,n\rightarrow\infty}C_{m,n}(u,v;h)\]
        \normalsize
        and the common limit is $C_{X_t,X_{t+h}}(u,v)$ (given by \eqref{cop}). Furthermore, the limits above are uniform in $(u,v)\in I^2$.
        \end{thm}

The proof of Theorem \ref{conv}, is a consequence of the following stronger lemma.

        \begin{lema}\label{lfund}
        Let $\{\mu_n\}_{n\in\N}$ be a sequence of  probability measures defined in $I$ such that $\mu_n\overset{w}{\longrightarrow}\mu$. Let $f_n:I\rightarrow I$ be a sequence of continuous functions converging uniformly to a function $f:I\rightarrow I$. Let $\{a_m\}_{m\in\N}$ be a sequence of real numbers such that $a_m\in[0,1]$ for all $m$ and $a_m\rightarrow a$. Also let $g_m:[a_m,1]\rightarrow I$ be a sequence of continuous functions converging uniformly to a function $g:I\rightarrow I$, $S_{m,n}(v):=\big[a_m,g_m\big(f_n(v)\big)\big]$ and $S(v):= \big[a,g\big(f(v)\big)\big]$. Then,
        \small
        \[\lim_{m\rightarrow\infty}\lim_{n\rightarrow\infty}\!\mu_n\big(S_{m,n}(v)\big)=\!\!\lim_{n\rightarrow\infty}
        \lim_{m\rightarrow\infty}\!\mu_n\big(S_{m,n}(v)\big)=\!\!\lim_{m,n\rightarrow\infty}\!\mu_n\big(S_{m,n}(v)\big)=\mu\big(S(v)\big)\]
        \normalsize
        uniformly in $v\in I$.
        \end{lema}
\begin{proof}
For all $m,n>0$ and $v\in[0,1]$, let $S_{m,n}(v)$ and $S(v)$ be as in the enunciate and let
\small
\[S_{n}(v):=\big[a,g\big(f_n(v)\big)\big]\quad\mbox{ and }\quad S_{m}(v):=\big[a_m,g_m\big(f(v)\big)\big] .\]
\normalsize
Notice that all sets just defined are $\mu$-continuity sets for all $m$, $n$ and $v$.
Since the convergence of $f_n$ to $f$ is uniform, we have
\small
\[\lim_{m,n\rightarrow\infty}g_m\big(f_n(v)\big)=\lim_{n\rightarrow\infty}\lim_{m\rightarrow\infty}g_m\big(f_n(v)\big)=\lim_{m\rightarrow\infty} \lim_{n\rightarrow\infty}g_m\big(f_n(v)\big)=g\big(f(v)\big)\]
\normalsize
for all $v$, so that, both, the iterated and the double limits exist and $S_{m,n}(v)\rightarrow S(v)$, for all $v\in[0,1]$. Also notice that we have $\I_{S_{m,n}}(x)\leq\I_I(x)$ uniformly in $m$, $n$ and $x$, and since $\mu_n$ converges weakly to $\mu$ and $I$ is a $\mu$-continuity set, it follows that
\small
\[\int\I_I(x)\mathrm{d}\mu_n\longrightarrow\int\I_I(x)\mathrm{d}\mu.\]
\normalsize
Now, in one hand, since $S_{m,n}(v)\rightarrow S_m(v)$ for all $v$ and $\delta_{S_{m,n}}\leq\delta_I$, by the Lebesgue convergence theorem, it follows that
\small
\[ \mu_n\big(S_{m,n}(v)\big)=\int \I_{S_{m,n}}(x) \mathrm{d}\mu_n\underset{n\rightarrow\infty}{\longrightarrow}
\int \I_{S_m}(x) \mathrm{d}\mu,\]
\normalsize
and, since $\delta_{S_m}\leq\delta_I$ and $\int\delta_I \mathrm{d}\mu<\infty$, by the Lebesgue dominated theorem, we conclude that
\small
\[\int \I_{S_m}(x) \mathrm{d}\mu\underset{m\rightarrow\infty}{\longrightarrow}\int \I_S(x) \mathrm{d}\mu =\mu\big(S(v)\big),\]
\normalsize
which shows that $\displaystyle{\lim_{m\rightarrow\infty}\lim_{n\rightarrow\infty}}\mu_n\big(S_{m,n}(v)\big)=\mu\big(S(v)\big)$
and the convergence holds uniformly in $v\in(0,1)$.
On the other hand, since $\delta_{S_{m,n}}\leq\delta_I$ and $\int \delta_I\mathrm{d}\mu_n<\infty$, by the Lebesgue dominated theorem, it follows that
\small
\[ \mu_n\big(S_{m,n}(v)\big)=\int \I_{S_{m,n}}(x) \mathrm{d}\mu_n\underset{m\rightarrow\infty}{\longrightarrow}
\int \I_{S_n}(x) \mathrm{d}\mu_n,\]
\normalsize
and, since $\delta_{S_n}\leq\delta_I$ and $\int \delta_I \mathrm{d}\mu_n \rightarrow\int \delta_I \mathrm{d}\mu$, by the Lebesgue convergence theorem we conclude that,
\small
\[\int \I_{S_n}(x) \mathrm{d}\mu_n\underset{n\rightarrow\infty}{\longrightarrow}\int \I_S(x) \mathrm{d}\mu =\mu\big(S(v)\big),\]
\normalsize
that is, $\displaystyle{\lim_{n\rightarrow\infty}\lim_{m\rightarrow\infty}}\mu_n\big(S_{m,n}(v)\big)=\mu\big(S(v)\big)$,
which also holds uniformly in $v$. Since the iterated limits are established, in order to finish the proof we need to show that the double limit exists and is equal to the iterated ones. Let $\eps>0$ be given. Since $\mu\ll\lambda$, the Radon-Nikodym theorem implies the existence of a non-negative continuous function $h$, which will be bounded since we are restricted to the interval $I$, such that, for any $A\in\mathcal{B}(I)$,
\small
\[\mu(A)=\int_A h(x)\mathrm{d}\lambda\leq M\lambda(A),\]
\normalsize
where \small$M=\displaystyle{\sup_{x\in I}\{h(x)\}}<\infty$. \normalsize Now, since $a_m\rightarrow a$, one can find $m_1:=m_1(\eps)>0$ such that, if $m>m_1$,
\small
\[a_m\in K_1(\eps):=\left[a-\frac{\eps}{10M},a+\frac{\eps}{10M}\right]\]
\normalsize
and
\small
\[\mu\big(K_1(\eps)\big)\leq M\lambda\left(\left[a-\frac{\eps}{10M},a+\frac{\eps}{10M}\right]\right)=\frac{\eps}{5}.\]
\normalsize
The uniform convergence of $g_m$ to $g$ implies the existence of $m_2:=m_2(\eps)>0$ such that, if $m>m_2$, $|g_m(x)-g(x)|<\eps/20M$, for all $x\in I$, or equivalently, taking $x=f_n(v)$, if $m>m_2$
\small
\[g_m\big(f_n(v)\big)\in \left[g\big(f_n(v)\big)-\frac{\eps}{20M}, g\big(f_n(v)\big)+\frac{\eps}{20M}\right].\]
\normalsize
Now, the uniform continuity of $g$ implies the existence of a $\delta:=\delta(\eps)>0$ such that
\small
\[\big|x-f_n(v)\big|<\delta \ \ \Longrightarrow \ \ \big|g(x)-g\big(f_n(v)\big)\big|<\frac{\eps}{20M}.\]
\normalsize
But since $f_n$ converges to $f$ uniformly, there exists a $n_1=n_1(\delta)>0$ such that
\small
\[n>n_1\ \ \Longrightarrow\ \ \big|f_n(v)-f(v)\big|<\delta,\]
\normalsize
for all $v$ so that, taking $x=f(v)$, for $n>n_1$, we have
\small
\[g\big(f_n(v)\big)\in\Big[g\big(f(v)\big)-\frac{\eps}{20M}, g\big(f(v)\big)+\frac{\eps}{20M}\Big],\]
\normalsize
for all $v\in I$. Hence, if we take $m>m_2$ and $n>n_1$,
\small
\[g\big(f_n(v)\big)-\frac{\eps}{20M}\in\left[g\big(f(v)\big)-\frac{\eps}{10M},g\big(f(v)\big)\right]\]
\normalsize
and
\small
\[g\big(f_n(v)\big)+\frac{\eps}{20M}\in\left[g\big(f(v)\big),g\big(f(v)\big)+\frac{\eps}{10M}\right]\]
\normalsize
so that,  setting
\small
\[K_2(\eps):= \left[g\big(f(v)\big)-\frac{\eps}{10M}, g\big(f(v)\big)+\frac{\eps}{10M}\right],\]
\normalsize
for $m>m_2$ and $n>n_1$, it follows that
\small
\[g_m\big(f_n(v)\big)\in\left[g\big(f_n(v)\big)-\frac{\eps}{20M}, g\big(f_n(v)\big)+\frac{\eps}{20M}\right]\subseteq K_2(\eps),\]
\normalsize
for all $v\in I$. Also observe that
\small
\[\mu\big(K_2(\eps)\big)\leq M\lambda\left(\left[g\big(f(v)\big)-\frac{\eps}{10M}, g\big(f(v)\big)+\frac{\eps}{10M}\right]\right)\leq \frac{\eps}{5}.\]
\normalsize
 The convergence of $\mu_n$ to $\mu$ implies the existence of $n_2:=n_2(\eps)>0$ such that if $n>n_2$ ($K_i(\eps)$ is a $\mu-$continuity set)
 \small
\[\big|\mu_n\big(K_i(\eps)\big)-\mu\big(K_i(\eps)\big)\big|< \frac{\eps}{5},\]
\normalsize
for $i=1,2$.
Also, if we set $F_n(x)=\mu_n\big([0,x]\big)$ and $F_0(x)=\mu\big([0,x]\big)$, then $F_0$ is continuous (since $\mu\ll\lambda$), $F_n\rightarrow F_0$, and, by Pólya's theorem, there exists a $n_3:=n_3(\eps)>0$ such that, if $n>n_3$
\small
\[\sup_{x\in I}\big\{\big|F_n(x)-F_0(x)\big|\big\}<\frac{\eps}{10}.\]
\normalsize
Now, notice that, if $n>n_3$
\small
\begin{eqnarray}
\big|\mu_n\big(S(v)\big)-\mu\big(S(v)\big)\big|\hspace{-.2cm}&\leq&\hspace{-.2cm}
\Big|F_n\big(g\big(f(v)\big)\big)-F_0\big(g\big(f(v)\big)\big)\Big|+\big|F_n(a)-F_0(a)\big|\nonumber\\
&\leq&\!\!2\sup_{x\in I}\big\{\big|F_n(x)-F_0(x)\big|\big\}<\frac{\eps}{5},\nonumber
\end{eqnarray}
\normalsize
for all $v\in I$. Observe further that, by construction, if $m>\max\{m_1,m_2\}$ and $n>n_1$,
\small
\[S_{m,n}(v)\backslash S(v)\subset K_1(\eps)\cup K_2(\eps),\]
\normalsize
for all $v$ so that, setting $n_0=n_0(\eps):=\max\{m_1,m_2,n_1,n_2,n_3\}$, if $m,n>n_0$, we have
\small
\begin{eqnarray*}
\big|&&\hspace{-1cm}\mu_n\big(S_{m,n}(v)\big)-\mu\big(S(v)\big)\big|\leq\\
&\leq&\!\!\!\big|\mu_n\big(S_{m,n}(v)\big)-\mu_n\big(S(v)\big)\big|+ \big|\mu_n\big(S(v)\big)-\mu\big(S(v)\big)\big|\\
&<&\!\!\!\big|\mu_n\big(K_1(\eps)\big)+\mu_n\big(K_2(\eps)\big)\big|+\frac{\eps}{5}\\
&\leq&\!\!\!\big|\mu_n\big(K_1(\eps)\big)\!-\!\mu\big(K_1(\eps)\big)\big|\!+\!\mu\big(K_1(\eps)\big)\!+\!\mu\big(K_2(\eps)\big)\!+\!
\big|\mu\big(K_2(\eps)\big)\!-\!\mu_n\big(K_2(\eps)\big) \big|\!+\!\frac{\eps}{5}\\
&<&\!\!\!\eps,
\end{eqnarray*}
\normalsize
for all $v$, which implies the existence of the double limit, equality with the iterated ones and the desired uniform convergence.\fim
\end{proof}

\noindent \emph{\textbf{Proof of Theorem \ref{conv}}:} First notice that taking $f_n=F_n^{-1}$, $g_m=\T_{h,k}^m$, $a_m=a_{h,k}^m$, it follows from Lemma \ref{lfund} that
\small
\[\mu_n\Big(\big[a_{h,k}^m,\T_{h,k}^m\big(F^{-1}_n(v)\big)\big]\Big)\underset{m,n\rightarrow\infty}{\longrightarrow}
\mu\Big(\big[a_{h,k},\T_{h,k}\big( F^{-1}_0(v)\big)\big]\Big),\]
\normalsize
for each $k=0,\cdots,n_0-1$. It remains to show that
\small
\[\lim_{m,n\rightarrow\infty}\mu_n\Big(\big[a_{h,n_0}^m,\min\big\{F_n^{-1}(u),\T_{h,n_0}^m\big( F^{-1}_n(v)\big)\big\}\big]\Big)=\mu\Big(\big[a_{h,n_0},\min\big\{F_0^{-1}(u),\T_{h,n_0}\big( F^{-1}_0(v)\big)\big\}\big]\Big),\]
\normalsize
and that the iterated limits exist and are equal to the double limit. First, since we can write $\min\{u,v\}=\frac{u+v}{2}-\frac{|u-v|}{2}$, it is routine to show that if  $f_n\rightarrow f$ uniformly, with $f_n$ and $f$ uniformly continuous and $g_m\rightarrow g$ uniformly, with $g_m$ and $g$ uniformly continuous, we have $\min\big\{f_n(u),g_m\big(f_n(v)\big)\big\}$ converging uniformly to $\min\big\{f(u),g\big(f(v)\big)\big\}$ in $n$, $m$, $u$ and $v$. So, the problem simplifies to show that if $a_m\rightarrow a$, $g_{m,n}(u,v)$ is a sequence of functions such that $g_{m,n}(u,v)\rightarrow g(u,v)$ uniformly in $u,v,n,m$ and $a_m\leq g_{m,n}(u,v)$ for all  $u,v,n,m$ and $\mu_n\overset{w}{\longrightarrow}\mu$, then
\small
\[\lim_{m,n\rightarrow\infty}\mu_n\big([a_m,g_{m,n}(u,v)]\big)=\mu\big([a,g(u,v)]\big),\]
\normalsize
uniformly in $u$ and $v$ and the double limit above is equal to the iterated limits. A similar argument to the one used in Lemma \ref{lfund} to establish the existence and equality of the iterated limits can be used to show the existence and equality of the iterated limits in this case. As for the double limit, let $M$ be as in the proof of Lemma \ref{lfund}. By the uniform convergence of $g_{m,n}(u,v)$ to $g(u,v)$ and since $g_{m,n}$ and $g$ are uniformly continuous for all $m,n$, it follows that there exists $m_1:=m_1(\eps)>0$, depending on $\eps$ only, such that, if $m,n>m_1$,
\small
\[g_{m,n}(u,v)\in K(\eps):=\left[g(u,v)-\frac{\eps}{10M}, g(u,v) +\frac{\eps}{10M}\right],\]
\normalsize
for all $u$ and $v$ and $\mu\big(K(\eps)\big)\leq \eps/5$. The rest of the proof is carried out by mimicking the proof of Lemma \ref{lfund} with the obvious adaptations. Identification of $g_{m,n}(u,v)$, $g(u,v)$, $a_m$ and $a$ with $\min\big\{F_n^{-1}(u),\T_{h,n_0}^m\big( F^{-1}_n(v)\big)\big\}$, $\min\big\{F_0^{-1}(u),\T_{h,n_0}\big( F^{-1}_0(v)\big)\big\}$, $a_{h,n_0}^m$ and $a_{h,n_0}$, respectively, completes the proof.\fim
    \begin{rmk}
    Notice that neither the convergence proved in Lemma \ref{lfund} nor the one in Theorem \ref{conv} is uniform in $m$ and $n$.
    \end{rmk}
As for the case when $\p$ is almost everywhere decreasing, we observe that, in view of \eqref{mm}, the function
\small
\[C_{m,n}^\ast(u,v;h)=u+v-1+C_{m,n}(1-u,1-v;h)\]
\normalsize
is an approximation to the copula in \eqref{copb}. Clearly $C_{m,n}^\ast$ converges to the true copula as $m$ and $n$ tends to infinity (view either as an iterated or a double limit) and the convergence is uniform in $(u,v)$.
\subsection*{Implementation and Random Variate Generation}
 The implementation of the approximations so far discussed is routine. All the approximations we mentioned can share the same iteration vector, which further improves the efficiency and precision of the task and greatly reduces the computational burden. In the top panel of Figure \ref{f2} we show the three dimensional plot of the lag 1 and 2 MP copula for values of $s\in\{0.1,0.4\}$. The respective level plots are shown in the bottom panel of Figure \ref{f2}. Notice the non-exchangeability of the copulas in all cases.
\begin{figure}[!h]
\centering
\mbox{
\includegraphics[width=0.22\textwidth]{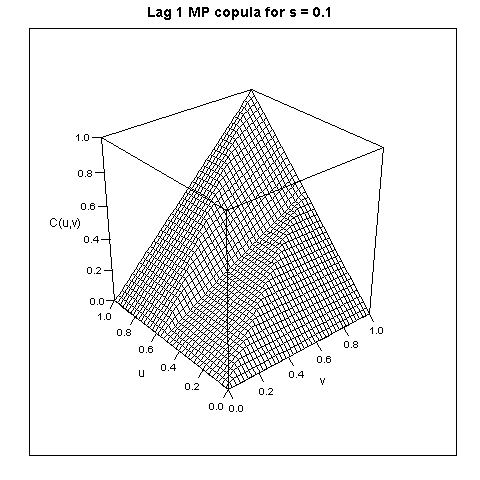}\hskip.3cm
\includegraphics[width=0.22\textwidth]{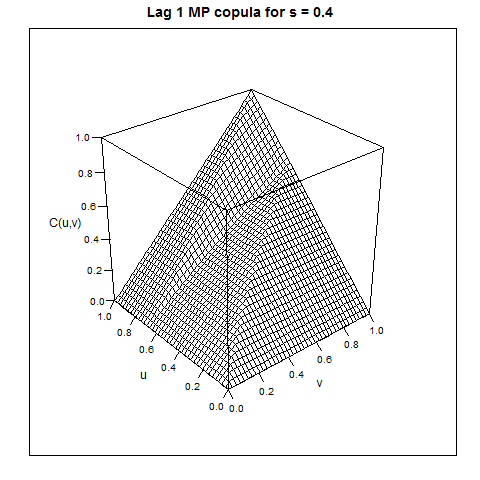}\hskip.3cm
\includegraphics[width=0.22\textwidth]{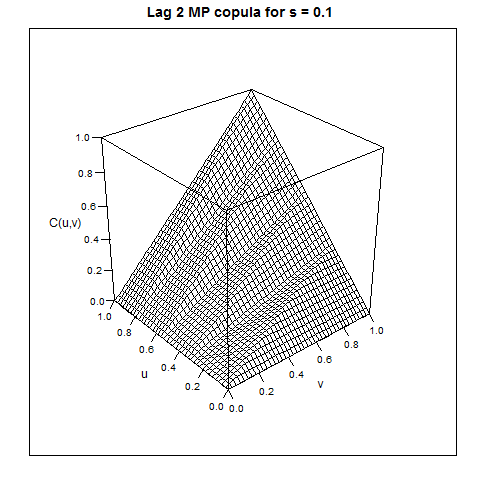}\hskip.3cm
\includegraphics[width=0.22\textwidth]{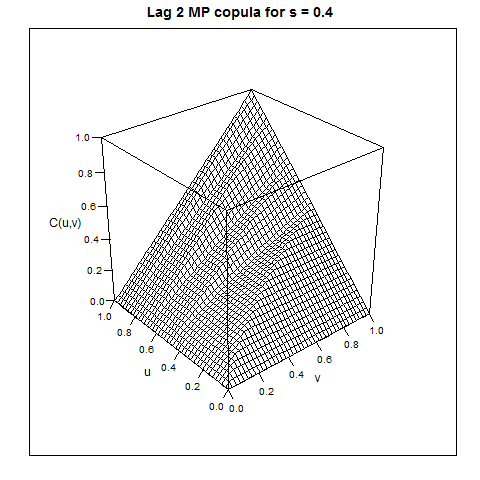}
  }
  \vskip.2cm
  \mbox{
\includegraphics[width=0.22\textwidth]{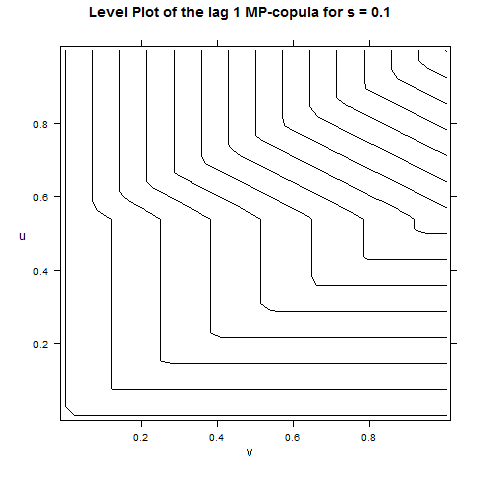}\hskip.3cm
\includegraphics[width=0.22\textwidth]{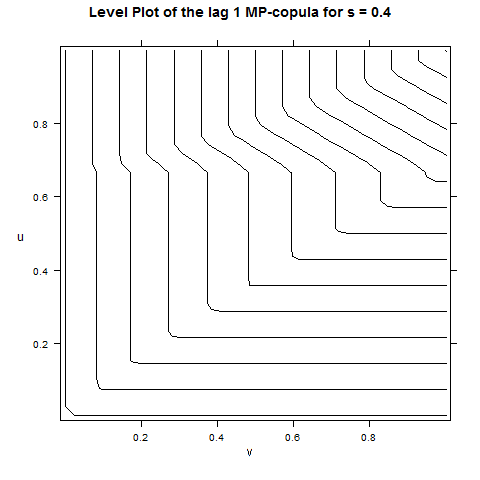}\hskip.3cm
\includegraphics[width=0.22\textwidth]{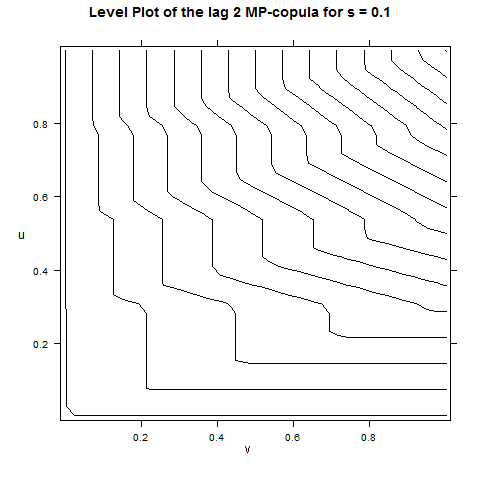}\hskip.3cm
\includegraphics[width=0.22\textwidth]{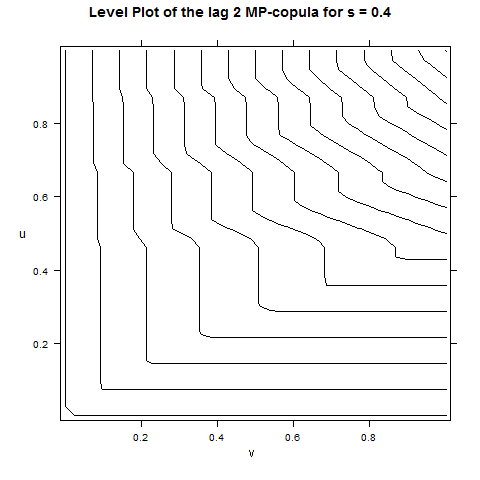}
}
\caption{From left to right, three dimensional plots of the lag 1 MP copula for $s\in\{0.1,0.4\}$ and lag 2 MP copula for the same parameters (top panel) and respective level sets (bottom panel) obtained from approximation \eqref{copapp}. }\label{f2}
\end{figure}

Obtaining random samples from an MP copulas is a trivial task in view of Proposition \ref{supp}. There we show that the support of an MP copula is the union of graphs of certain linear functions. The following algorithm can be used to generate a pair of variates from a bidimensional MP copula  for $\p$ an almost everywhere increasing function.
\begin{itemize}
\item[1.] Generate an uniform $(0,1)$ variate $u$.
\item[2.] Let $\kappa_0$ denote the index for which $u\in\big[F_0(a_{h,\kappa_0}),F_0(a_{h,\kappa_0+1})\big]$ and set $v=\ell^+_{h,\kappa_0}(u)$.
\item[3.] The desired pair is $(u,v)$.
\end{itemize}
In practice the $T_s$-invariant probability measure is unknown and $F_0$ has to be approximated. Furthermore, most of times the nodes related to $T^h_s$, for $h>0$, $s\in(0,1)$ cannot be analytically obtained. However, we can apply the approximations developed in this section together with the algorithm above to obtain approximated samples from MP copulas. In Figure \ref{samp} we show 500 approximated sample points from a lag 1 and 2 MP copula for  $s\in\{0.1,0.4\}$ and $\p$ an almost everywhere. Obvious modifications in the algorithm, allow handling the case where $\p$ is an almost everywhere decreasing function.

\begin{figure}[!ht]
\centering
\mbox{
\includegraphics[width=0.23\textwidth]{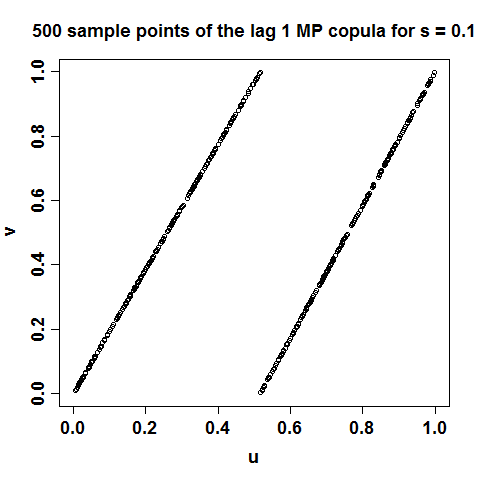}\hskip.3cm
\includegraphics[width=0.23\textwidth]{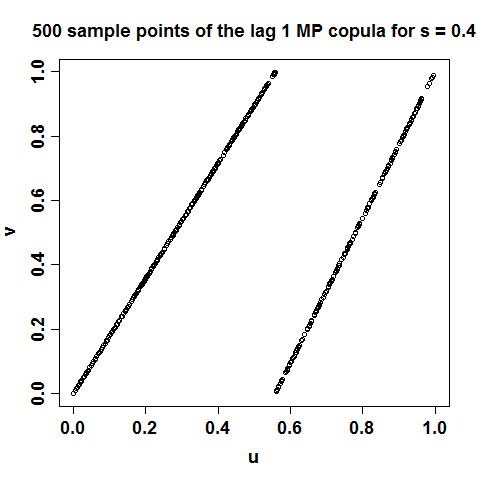}\hskip.3cm
\includegraphics[width=0.23\textwidth]{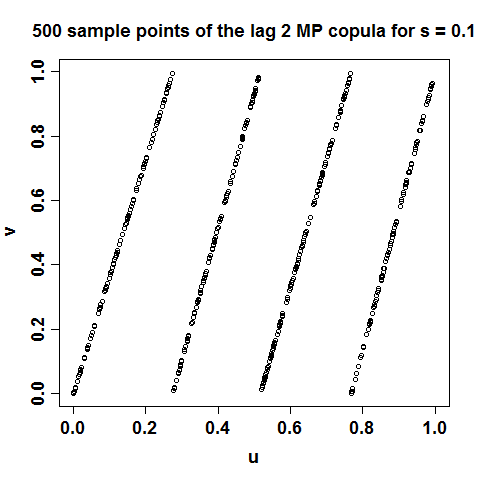}\hskip.3cm
\includegraphics[width=0.23\textwidth]{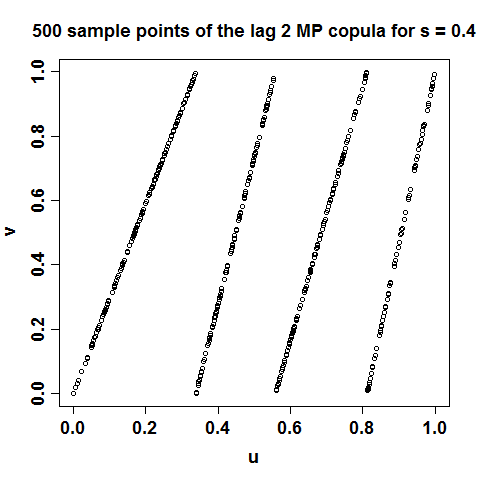}
  }
  \caption{Left to right: 500 approximated sample points from  a lag 1 MP copula for $s\in\{0.1,0.4\}$ and lag 2 MP copula for the same parameters. }\label{samp}
\end{figure}
\begin{rmk}
For small values of the lag, the resemblance of the sample to a piecewise continuous function is very clear, but this is not always the case as it can be seen in Figure \ref{special}, where we show 500 approximated sample points of the lag 4, 5 and 7 MP copulas for $s=0.2$. This is a general principle, for a fixed sample size the higher the lag, the harder to distinguish the support of the copula based on the sample, since the number of branches of $T_s^h$ grow as fast as $2^h$. For instance, for $h=7$ in Figure \ref{special} is difficult to say that the sample came from a singular copula at all.
\end{rmk}
\begin{figure}[!h]
\centering
\mbox{
\includegraphics[width=0.23\textwidth]{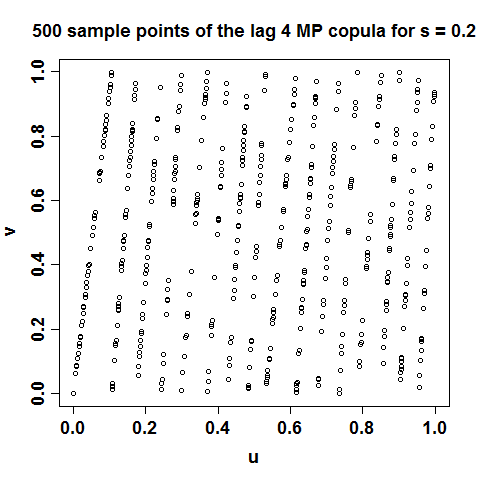}\hskip1cm
\includegraphics[width=0.23\textwidth]{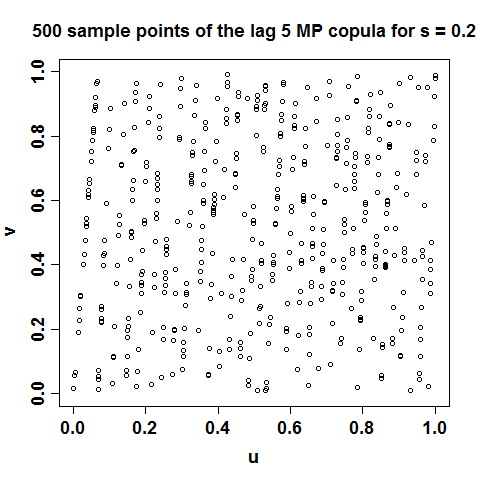}\hskip1cm
\includegraphics[width=0.23\textwidth]{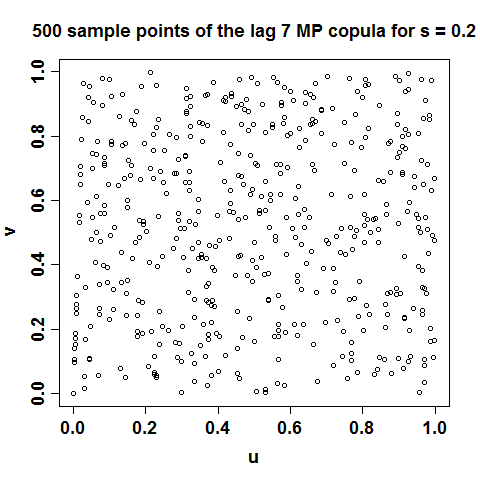}
  }
  \caption{Left to right: 500 approximated sample points from the lag 4, 5 and 7 MP copulas for $s=0.2$. }\label{special}
\end{figure}

\section{Application}

In this section we apply the theory developed in Section 3 to the problem of estimating the parameter $s$ in  MP processes. This problem have been studied before in Olbermann et. al (2007), where the authors adapt and apply several estimation methods  from the classical theory of long-range dependence to the problem of estimating the parameter $s$.  In this section we propose an estimator for the parameter $s$ based on the ideas developed in Section 3, which is both, precise and fast.

The mathematical framework is as follows. Let $s\in(0,1)$ and consider the associated MP process $\{X_n\}_{n\in\N}$ for $\phi$ the identity map. Suppose we observe a realization $x_1,\cdots,x_N$ from $X_n$ and our goal is to estimate the unknown  parameter $s$. Let $a:=a(s)\in\big(\frac12,\frac{\sqrt{5}-1}{2}\big)$ denote the discontinuity point of the MP transformation and notice that $s$ and $a$ are related by
\small
\[a+a^{1+s}=1\quad \Longleftrightarrow \quad s=\frac{\log(1-a)}{\log(a)}-1.\]
\normalsize
Hence, the problem of estimating $s$ is equivalent to the problem of estimating $a$.

To define the proposed estimator, we start by observing that Proposition \ref{supp} for $h=1$ implies that the lag 1 MP copula's support is given by the graph of the piecewise linear function
\small
\[\ell(x):=\left\{\begin{array}{cc}
                   \frac{x}{F_0(a)}\,, & \text{ if } x\in\big[0,F_0(a)\big) \\
                   \frac{x-F_0(a)}{1-F_0(a)}\,, & \text{ if } x\in\big[F_0(a),1].
                 \end{array}
\right.\]
\normalsize
so that, any (independent or correlated) sample from a lag 1 MP copula consists of points scattered through the lines defined by $\ell$ (see Figure \ref{samp}). The discontinuity point of the function $\ell$ is precisely $F_0(a)$. Let $y_i=F_0(x_i)$, for $i=1,\cdots,N$, and consider the series $\{ u_i:=(y_i,y_{i+1})\}_{i=1}^{N-1}$. By Sklar's Theorem, $\{u_i\}_{i=1}^{N-1}$ is a (correlated) sample  from the lag 1 MP copula, so all points should lie in the graph of the function $\ell$.

These considerations suggest the following procedure to obtain $s$ based on a path $x_1,\cdots,x_N$ of $X_n$ within a given accuracy $\eps>0$. We choose $s_0\in(0,1)$ as an initial guess for $s$ and calculate $\hat y_i=F_n(x_i;s_0)$, $i=1,\cdots,N$, where $F_n$ is the approximation of $F_0$ given in \eqref{fn}. Next we define $\{ \hat u_i:=(\hat y_i,\hat y_{i+1})\}_{i=1}^{N-1}$, from which we estimate the slope of the two branches of the approximated sample from the lag 1 MP copula obtained by this way. The discontinuity point (and hence $s$) can then be easily calculated. In this manner we obtain an estimative $\tilde s$ which can be compared to $s_0$.  If $s_0$ is close to the true value $s$, then the difference between $\tilde s$ and $s_0$ should be small. If not, we choose another starting value and repeat the operation until obtain the desired accuracy. This leads to an optimization procedure to obtain $s$ within a predefined accuracy.

 To illustrate the procedure, Figure \ref{spec}(a) shows a sample path of an MP process for $s=0.2$, with $N=200$ while Figure \ref{spec}(b) shows the sample path $y_i=F_n(x_i;0.2)$, $i=1,\cdots,N$. From $\{y_i\}_{i=1}^N$, we construct the sequence $\{u_i\}_{i=1}^{N-1}$, where $u_i=\big(F_n(y_i;s),F_n(y_{i+1};s)\big)$, $i=1,\cdots,N-1$, for the correctly specified $s=0.2$ and for $s=0.3$. Figure \ref{spec}(c) presents the graph of $\{u_i\}_{i=1}^{N-1}$ obtained from the correct specification of $s$, while Figure \ref{spec}(d) shows the graph of the misspecified one. In Figures \ref{spec}(c) and \ref{spec}(d), the solid lines represent the respective theoretical support of the copula given in Proposition \ref{supp}. Some distortion in the points can be seen given to the use of the approximation $F_n$ instead of the theoretical $F_0$, especially in lower quantiles. From Figure \ref{spec}(d) it is clear that the line obtained from the sequence $\{u_i\}_{i=1}^{N-1}$ and the theoretical one for the chosen value of $s_0$, namely, 0.3, do not match, while for the correct specified one in Figure \ref{spec}(c), they do.

\begin{figure}[!h]
\centering
\mbox{
\subfigure[]{\includegraphics[width=0.23\textwidth]{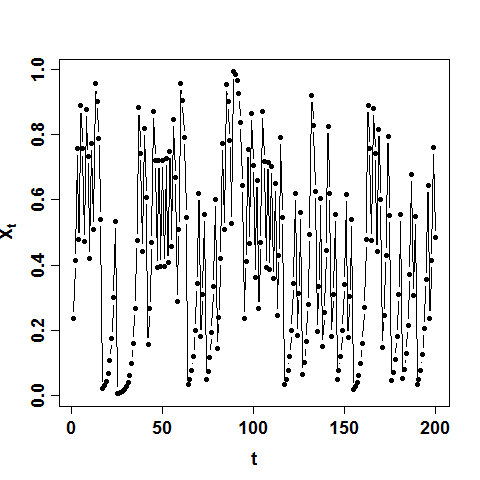}}\hskip.2cm
\subfigure[]{\includegraphics[width=0.23\textwidth]{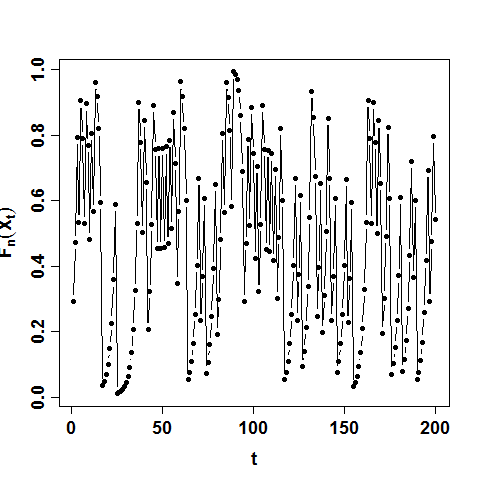}}\hskip.2cm
\subfigure[]{\includegraphics[width=0.23\textwidth]{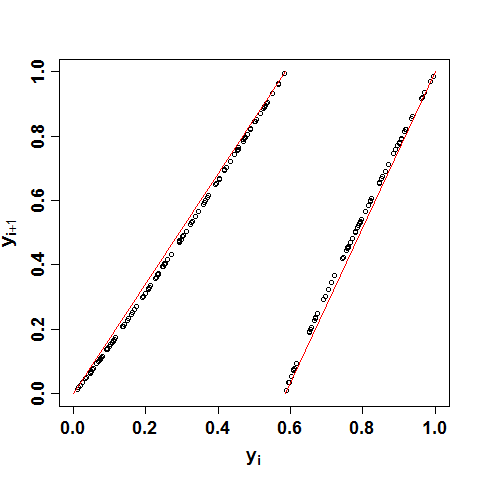}}\hskip.2cm
\subfigure[]{\includegraphics[width=0.23\textwidth]{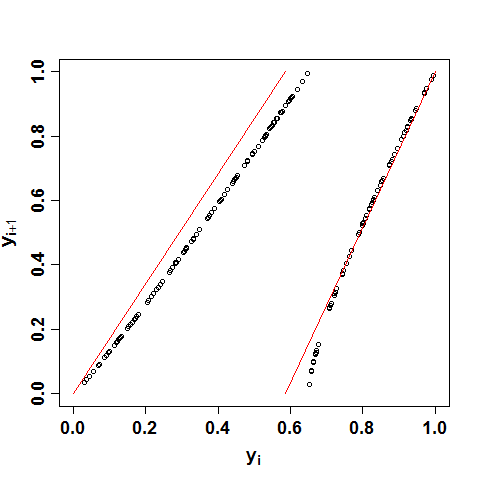}}
}
  \caption{(a) Sample path $x_1,\cdots,x_{200}$ of an MP process with $s=0.2$ starting at $\sqrt{5}(\mathrm{mod}1)$. (b) Sample path $y_i=F_n(x_i)$. Plot of $u_i=(y_i,y_{i+1})$ for the (c) correct  and (d) misspecified $s$. The solid lines correspond to the theoretical support of the respective lag 1 MP copula. }\label{spec}
\end{figure}

The procedure just outlined is, however, computationally expensive given the fact that to calculate the approximation $F_n$ with reasonable stability and accuracy, for each $s$, it requires the construction of an iteration vector of large size (see Figure \ref{f1} and Table \ref{t1}). Such an optimization procedure can easily take hundreds of evaluations, depending on the desired accuracy, and hence, can be a very time consuming task.

To overcome this difficulty, observe that in Figures \ref{spec}(a) and \ref{spec}(b), little differences can be seen between them. In fact, since $F_0$ is a smooth distribution, an alternative is to apply the previous argument to the points $\hat v_i:=(x_i,x_{i+1})$, $i=1,\cdots,N$. There will certainly be some distortion in the lines due to the absence of $F_0$, but we expect to be able to estimate the discontinuity point $a$ based on $v_i$ by similar idea as before.

As an illustration, Figure \ref{spec2} shows the plots of $v_i=(x_i,x_{i+1})$, $i=1,\cdots,199$, based on MP processes with $s\in\{0.2,0.4,0.6,0.8\}$ all starting at $\sqrt{5}(\mathrm{mod}1)$. The solid lines correspond to the lines joining the points $(0,0)$ and $(a,1)$ and joining $(a,0)$ and $(1,1)$, where $a$ denotes the correct discontinuity point of the respective MP transformation. From the graphs in Figure \ref{spec2} we see the identification of the line based on $v_i$ with the correct line, especially in the second branch of the graph. That is so because $a\in\big(\frac12,\frac{\sqrt{5}-1}{2}\big)$, so that the second branch, being smaller, is less affected by the distortion due to the absence of $F_0$.

\begin{figure}[!h]
\centering
\mbox{
\subfigure[]{\includegraphics[width=0.23\textwidth]{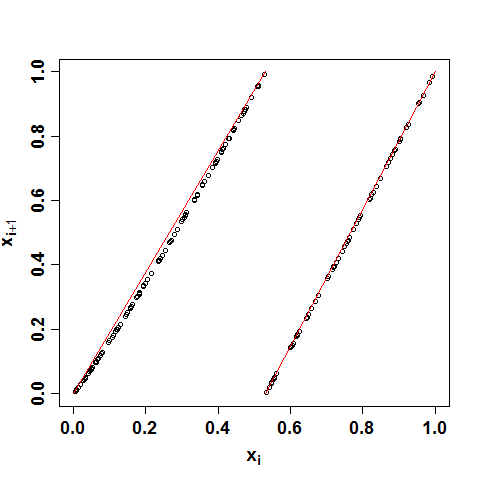}}\hskip.2cm
\subfigure[]{\includegraphics[width=0.23\textwidth]{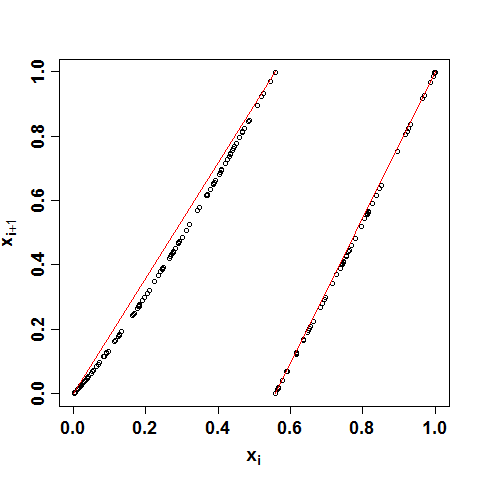}}\hskip.2cm
\subfigure[]{\includegraphics[width=0.23\textwidth]{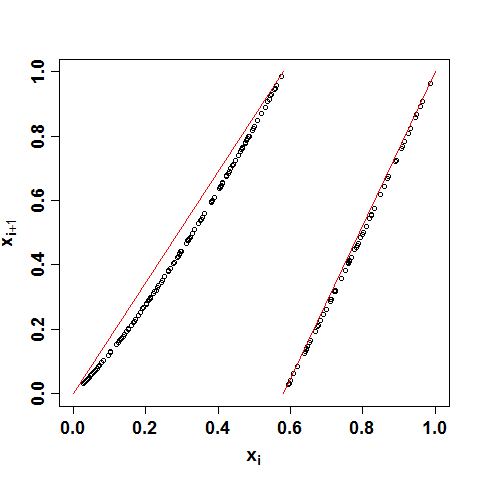}}\hskip.2cm
\subfigure[]{\includegraphics[width=0.23\textwidth]{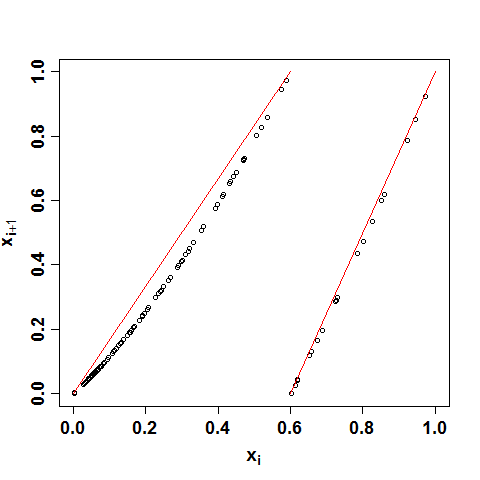}}
}
  \caption{Plot of $v_i=(x_i,x_{i+1})$, $i=1,\cdots,199$ from a sample path of an MP process with (a) $s=0.2$, (b) $s=0.4$, (c) $s=0.6$ and (d) $s=0.8$. The solid lines correspond to the lines joining the points $(0,0)$ and $(a,1)$ and joining $(a,0)$ and $(1,1)$, where $a$ denotes the correct discontinuity point of the respective MP transformation. }\label{spec2}
\end{figure}

In order to assess the performance of the estimation procedure, we perform the following experiment. We randomly select 100 initial points\footnote{Tables with the initial values applied in our experiments and the complete simulation results are available upon request.} in $(0,1)$ and for each initial point we generate a path (of size $N=200$) of an MP process for $s\in\{0.1,0.15,\cdots,0.95\}$. For each path, say $x_1,\cdots,x_{200}$, we perform the proposed estimation procedure. In order to estimate $a$, we applied two methods: the first one is a simple least squares method applied to the points lying in the second branch of $(x_i,x_{i+1})$. The second method is the following: let $(x_{m_0},x_{m_0+1})$ and $(x_{m_1},x_{m_1+1})$ denote the points among the ones lying on the second branch of $\{(x_i,x_{i+1})\}_{i=1}^{N-1}$ for which $x_{m_0}$ is minimum and $x_{m_1}$ is maximum. We define the estimator of $a$, say $\hat a$, as
\small
\begin{equation}\label{aest}
\hat{a}=-\frac BA\,,\quad\text{ where } \quad A:=\frac{x_{m_1+1}-x_{m_0+1}}{x_{m_1}-x_{m_0}}\quad\mbox{ and }\quad B:=x_{m_0+1}-Ax_{m_0}.
\end{equation}
\normalsize
For reference, in the subsequent we shall call this the \emph{min-max procedure}. Geometrically, $\hat a$ is the inverse image of 0 by the linear function joining $(x_{m_0},x_{m_0+1})$ and $(x_{m_1},x_{m_1+1})$.

\begin{figure}[!h]
\centering
\mbox{
\subfigure[]{\includegraphics[width=0.23\textwidth]{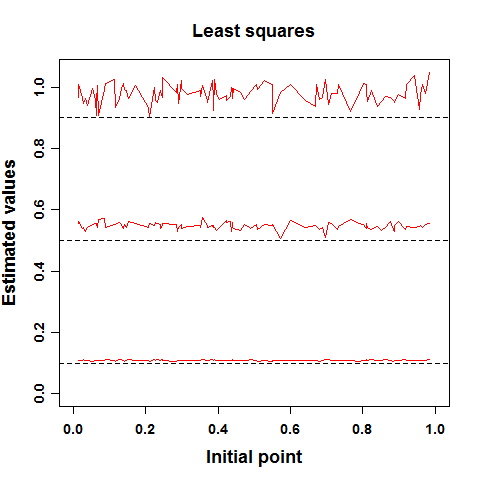}}\hskip.2cm
\subfigure[]{\includegraphics[width=0.23\textwidth]{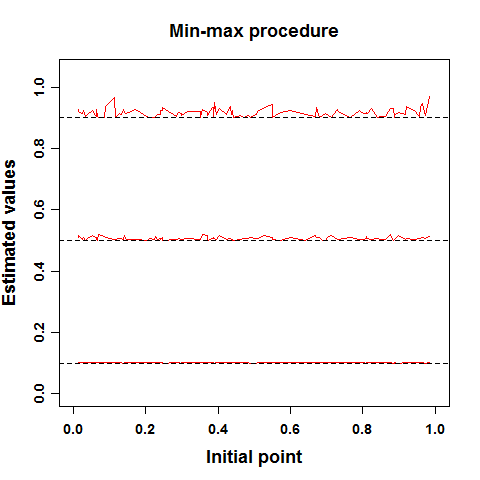}}\hskip.2cm
\subfigure[]{\includegraphics[width=0.23\textwidth]{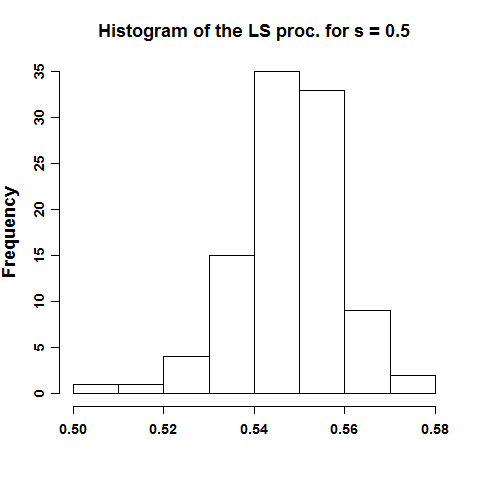}}\hskip.2cm
\subfigure[]{\includegraphics[width=0.23\textwidth]{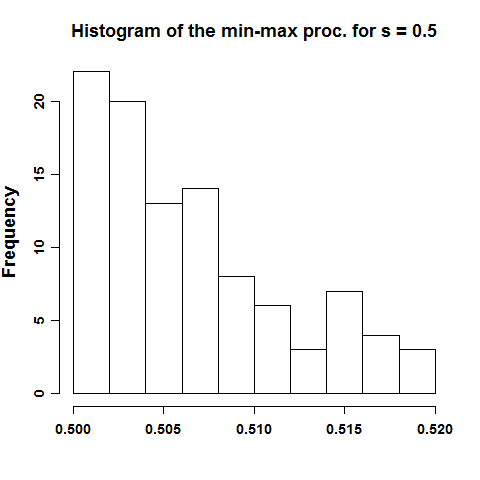}}
}
  \caption{Plot of the estimated values for $s\in\{0.1,0.5,0.9\}$ for 100 random initial points by using (a) the least squares procedure and (b) the min-max procedure. The dashed lines correspond to the correct value of $s$. Also shown the histogram of the estimated values for $s=0.5$ by using (c) the least squares procedure and (d) the min-max procedure.}\label{spec3}
\end{figure}

Table \ref{table1} summarizes the experiment results by presenting the mean, range, standard deviation (st.d.) and mean square error (mse) of the results. Figures \ref{spec3}(a) and \ref{spec3}(b) present graphically the results for both methods for $s\in\{0.1,0.5,0.9\}$ while in Figures \ref{spec3}(c) and \ref{spec3}(d), the histogram of the results for $s=0.5$ are presented. From Table \ref{table1} and Figure \ref{spec3}, we see that the  min-max procedure (MM) outperforms the least squares estimates (LS) obtained. Some bias can be seen for both estimates, especially when $s$ increases.

\begin{table}[!h]
\renewcommand{\arraystretch}{1.2}
\setlength{\tabcolsep}{3pt}
\caption{Summary statistics of the experiment results. Presented are the mean estimate ($\hat s$), the range, standard deviation (st.d.) and the mean square error (mse) of the estimates. The min-max procedure is denoted by MM while LS denotes the least squares.}\label{table1}
\vskip.3cm
\centering
{\footnotesize
\begin{tabular}{|c||c|c|c|c|c||c|c|c|c|c|}
\hline
Proc.&$s$&$\hat{s}$&range&st.d.&mse&$s$&$\hat{s}$&range&st.d.&mse\\
\hline\hline
MM  &\multirow{2}{*}{0.10}    	 &   	0.1008	&   $[0.1000,0.1024]$    &	0.0006	&   $0^\ast$&	\multirow{2}{*}{0.55 }	  &   	0.5581	&   $[0.5500,0.5888]$    &	0.0069	 &  0.0001\\
 LS &   	  &   	0.1087	&   $[0.1056,0.1128]$    &	0.0017	&   0.0001&	    	 &   	0.6036	&   $[0.5501,0.6287]$    &	0.0142&	   0.0031\\
\hline							
MM &\multirow{2}{*}{0.15} 	  &     	0.1516	&$[ 0.1500,0.1597]$    &	0.0015	&   $0^\ast$&	\multirow{2}{*}{0.60  } 	  &  	0.6091	&   $[0.6000,0.6451]$    &	0.0090	  & 0.0002\\
LS &   	 &   	0.1632	&   $[0.1573,0.1710]$    &	0.0026	&   0.0002&    	 &  	0.6545	&   $[0.6012,0.7023]$    &	0.0186	&   0.0033\\
\hline							
MM &\multirow{2}{*}{0.20} 	  &     	0.2023	&  $[0.2001,0.2101]$    &	0.0021&	   $0^\ast$&	\multirow{2}{*}{0.65}	  &      	0.6600	&   $[0.6501,0.6927]$    &	0.0087	&   0.0002\\
LS &    	 &   	0.2179	&   $[0.2098,0.2315]$    &	0.0038	&   0.0003&	    	 &   	0.7089	&   $[0.6579,0.7584]$    &	0.0197	&   0.0039\\
\hline					
MM &\multirow{2}{*}{0.25}  	  &    	0.2534	&  $[ 0.2501,0.2632]$    &	0.0027	 &  $0^\ast$&	\multirow{2}{*}{0.70    }	  & 	0.7125	&  $[ 0.7001,0.7726]$    &	0.0119	  & 0.0003\\
LS &    	 &   	0.2730	&   $[0.2636,0.2875]$    &	0.0049	&   0.0006&   	&  	0.7646	&  $[ 0.7038,0.8170]$    &	0.0226	&   0.0047\\
\hline						
 MM &\multirow{2}{*}{0.30  }	  &   	0.3036	&  $[ 0.3000,0.3128]$    &	0.0028	&   $0^\ast$&	\multirow{2}{*}{0.75   }	  &   	0.7621	&  $[ 0.7502,0.8019]$    &	0.0110&	   0.0003\\
 LS & 	    	  & 0.3272	&   $[0.3128,0.3410]$    &	0.0052	&   0.0008&	    	 &  	0.8177	&   $[0.7505,0.8612]$    &	0.0246	&   0.0052\\
\hline				
  MM &\multirow{2}{*}{0.35 }	  &   	0.3544	&  $[ 0.3500,0.3669]$    &	0.0039	&   $0^\ast$&	\multirow{2}{*}{0.80    }	  &  0.8165	&   $[0.8001,0.8659]$    &	0.0131	&   0.0004\\
  LS & 	   	 & 0.3835	&   $[0.3550,0.4024]$    &	0.0078	&   0.0012&	    	  &   	0.8781	&   $[0.8005,0.9449]$    &	0.0277	&   0.0069\\
\hline			
  MM &\multirow{2}{*}{0.40    }	  &	0.4050	&   $[0.4000,0.4214]$    &	0.0049	&   $0^\ast$&	\multirow{2}{*}{0.85}	  &     	0.8677	&  $[ 0.8500,0.9428]$    &	0.0151	&   0.0005\\
   LS &  	   &  	0.4367	&   $[0.4082,0.4570]$    &	0.0078	&   0.0014&	    	 &   	0.9307	&   $[0.8507,1.0111]$    &	0.0297	&   0.0074\\
\hline					
 MM &\multirow{2}{*}{0.45 }	  &    	0.4556	&   $[0.4501,0.4703]$    &	0.0046	  & 0.0001&	\multirow{2}{*}{0.90 } 	  &  	0.9172	&   $[0.9002,0.9704]$    &	0.0141&	   0.0005\\
LS &    	 &   	0.4909	&   $[0.4702,0.5174]$    &	0.0102	&   0.0018&   	 &   	0.9774	&   $[0.9011,1.0477]$    &	0.0306	&   0.0069\\
\hline					
MM	  & \multirow{2}{*}{0.50  }  &  	0.5065	&   $[0.5001,0.5189]$    &	0.0051	&   0.0001&	\multirow{2}{*}{0.95} 	  & 	0.9706	&   $[0.9500,1.0517]$    &	0.0189	&   0.0008\\
 LS &    	 &  	0.5475	&   $[0.5059,0.5746]$    &	0.0111	 &  0.0024&	    	  &   	1.0371	&   $[0.9500,1.1641]$    &	0.0384	&   0.0090\\
\hline											
\multicolumn{11}{l}{{\scriptsize Note: $0^\ast$ means that the mse is smaller than $5\times10^{-5}$.}}
\end{tabular}}
\end{table}

The min-max procedure can be carried out even for time series of sample size as small as 20, as long as the second branch of $\{(x_i,x_{i+1})\}_{i=1}^{N-1}$ contains at least 2 points, which does not always happen (for instance, for $N=110$, a sample path of an MP processes with $s=0.8$ starting at $\sqrt{74}(\mathrm{mod}1)$ has only one point in the second branch). In such a situation, a straightforward adaptation of the min-max procedure can be applied to the first branch and still yields reasonable estimates. The closer to 0 and 1 the points $x_{m_0}$ and $x_{m_1}$ in \eqref{aest} are, respectively, the better the estimation performance.

\FloatBarrier
\section{Conclusions}
\indent In this work we derive the copulas related to Manneville-Pomeau processes for almost everywhere monotonic functions $\p$. In the bidimensional case, we find that the copulas of any random pair $(X_t,X_{t+h})$ depend only on the lag $h$ and are singular. The support of the copulas is derived as well.

As for the multidimensional case,  when $\p$ is increasing almost everywhere, the functional form of the copulas are very similar to the ones in derived in the bidimensional case. We conclude that the copulas of vectors {\small$(X_{t_1},\cdots,X_{t_n})$} and {\small$\big(U_0,T^{t_2-t_1}_s(U_0),\cdots,T^{t_n-t_1}_s(U_0)\big)$} are the same. When $\p$ is decreasing almost everywhere, we find that the copulas of an $n$-dimensional random vector from an MP process can be deduced from the ones derived for the increasing case.

The copulas derived here depend on the $T_s$-invariant measure $\mu_s$ which has no explicit formula.  For the bidimensional case, we propose an approximation to the copula which is shown to converge uniformly to the true copula. From this approximation, we are able to present plots of the copulas for different parameters and lags and to present a simple algorithm to generate approximated samples from the copulas. Some simple numerical calculation are presented to test the steps of the approximation. To illustrate the usefulness of the theory, we derive  a fast estimation procedure of the underlying parameter $s$ in Manneville-Pomeau processes.

\subsubsection*{Acknowledgements}
\small
Sílvia R.C. Lopes research was partially supported by CNPq-Brazil, by CAPES-Brazil, by INCT {\it em Matem\'atica} and also by Pronex {\it Probabilidade e Processos Estoc\'asticos} - E-26/170.008/2008 -APQ1. Guilherme Pumi was partially supported by CAPES/Fulbright Grant BEX 2910/06-3 and by CNPq-Brazil.
\normalsize
\subsubsection*{References}
\small
\begin{description}
\item[] Billingsley, P. (1999). {\slshape Convergence of Probability Measures.} 2nd ed., New York: John Wiley. MR1700749
\item[] Chazottes, J.-R., P. Collet and B. Schmitt (2005). ``Statistical Consequences of the Devroye Inequality
for Processes. Applications to a Class of Non-Uniformly Hyperbolic Dynamical Systems''. {\slshape Nonlinearity},
\textbf{18(5)}, 2341-2364. MR2165706
\item[] Dellnitz, M. and Junge, O. (1999). ``On the Approximation of Complicated Dynamical Behavior''. {\slshape SIAM Journal of Numerical Analysis}, Vol. \textbf{36(2)}, 491-515. MR1668207
\item[] Fisher, A.M. and Lopes, A. (2001). ``Exact Bounds for the Polynomial Decay of Correlation, $1/f$ Noise and the CLT for the Equilibrium State of a Non-H\"{o}lder Potential''. {\slshape Nonlinearity}, Vol. \textbf{14}, 1071-1104. MR1862813
\item[] Joe, H. (1997). {\slshape Multivariate Models and Dependence Concepts.} Monographs on Statistics and Applied Probability, 73. London: Chapman $\&$ Hall. MR1462613
\item[] Koles\'arov\'a, A.; Mesiar, R.; Sempi, C. (2008). ``Measure-preserving transformations, copul{\ae} and compatibility''. {\slshape Mediterranean  Journal of Mathematics}, \textbf{5(3)}, 325-339. MR2465579
\item[] Lopes, A. and Lopes, S.R.C. (1998). ``Parametric Estimation and Spectral Analysis of Piecewise Linear Maps of the Interval''. {\slshape Advances in Applied Probability}, Vol. \textbf{30}, 757-776. MR1663557
\item[] Maes, C.; Redig, F.; Takens, F.;  Moffaert, A. and Verbitski, E. (2000). ``Intermittency and Weak Gibbs States''. {\slshape Nonlinearity}, Vol. \textbf{13}, 1681-1698. MR1781814
\item[] Nelsen, R.B. (2006). {\slshape An Introduction to Copulas}. New York: Springer-Verlag. 2nd Edition. MR2197664
\item[] Olberman, B.P; Lopes, S.R.C and Lopes, A.O. (2007). ``Parameter Estimation in Manneville-Pomeau Processes''. Unpublished manuscript. Source: arXiv:0707.1600.
\item[] Pianigiani, G. (1980). ``First Return Map and Invariant Measures''. {\slshape Israel Journal of Mathematics}, Vol. \textbf{35}, 32-48. MR0576460
\item[] Pollicott, M. and Yuri, M. (1998). {\slshape Dynamical Systems and Ergodic Theory}. Cambridge: Cambridge University Press. MR1627681
\item[] Royden, H.L. (1988). {\slshape Real Analysis.} New York: Macmillan. 3nd Edition. MR1013117
\item[]  Schweizer, B. and Sklar, A. (2005). {\slshape Probabilistic Metric Spaces}. Mineola: Dover Publications. MR0790314
\item[] Young, L-S. (1999).  ``Recurrence Times and Rates of Mixing''. {\slshape Israel Journal of Mathematics}, Vol. \textbf{110}, 153-188. MR1750438
\item[] Zebrowsky, J.J. (2001). ``Intermittency in Human Heart Rate Variability''. {\slshape Acta Physica Polonica B}, Vol. \textbf{32}, 1531-1540.
\end{description}
\end{document}